\documentclass[letterpaper,11pt]{amsart}

\usepackage{hyperref}

\usepackage{amssymb}
\usepackage{wasysym}
\usepackage{stix}
\usepackage{amsmath}
\usepackage{amsthm}
\usepackage{enumerate}
\usepackage{graphicx}
\usepackage{tikz}
\usetikzlibrary{arrows}
\usepackage{caption}
\usepackage{subcaption}

\theoremstyle{plain}
\newtheorem{thm}{Theorem}[section]
\newtheorem{prop}[thm]{Proposition}
\newtheorem{lem}[thm]{Lemma}
\newtheorem{cor}[thm]{Corollary}

\newtheorem{conj}[thm]{Conjecture}

\numberwithin{equation}{section}

\theoremstyle{definition}

\theoremstyle{remark}
\newtheorem*{remark}{Remark}

\newtheorem*{acknowledgements}{Acknowledgements}

\newcommand{\what}{\widehat}
\newcommand{\wtilde}{\widetilde}
\newcommand{\RR}{\mathbb{R}}
\newcommand{\Sp}{\mathbb{S}}
\newcommand{\DD}{\mathbb{D}}
\newcommand{\HH}{\mathbb{H}}
\newcommand{\CC}{\mathbb{C}}
\newcommand{\CHAT}{\what{\mathbb{C}}}
\newcommand{\ZZ}{\mathbb{Z}}

\newcommand{\g}{g} 
\newcommand{\h}{h} 
\newcommand{\origin}{X_0} 

\newcommand{\Hbar}{\overline{\mathbb{H}} \cup \{\infty\}}

\newcommand{\mf}{\mathcal{MF}} 
\newcommand{\pmf}{\mathcal{PMF}} 

\DeclareMathOperator{\im}{Im}
\DeclareMathOperator{\teich}{\mathcal{T}}
\DeclareMathOperator{\el}{EL} 
\DeclareMathOperator{\divr}{div} 
\DeclareMathOperator{\area}{Area}
\DeclareMathOperator{\fix}{Fix}
\DeclareMathOperator{\crit}{Crit}

\begin{document}

\title[Toy Teichm\"uller spaces of real dimension 2]{Toy Teichm\"uller spaces of real dimension 2:\\ the pentagon and the punctured triangle}

\author[Y. Chen]{Yudong Chen}
\address{Faculty of Mathematics,
University of Cambridge, Centre for Mathematical Sciences, Wilberforce Road,
Cambridge CB3 0WA,
United Kingdom}
\email{eason.chenyudong@gmail.com}

\author[R. Chernov]{Roman Chernov}
\address{School Mathematics and Computer Sciences, V. N. Karazin Kharkiv
National University, 4 Svobody Sq., Kharkiv, 61022, Ukraine}
\email{roman1chernov@gmail.com}

\author[M. Flores]{Marco Flores}
\address{Departamento de Matem\'aticas, Universidad de Guanajuato, 
Jalisco s/n Mineral de Valenciana, C.P. 36240 
Guanajuato, Gto. M\'exico}
\email{marco.flores@cimat.mx}

\author[M. F. Bourque]{Maxime Fortier Bourque}
\address{Department of Mathematics, University of Toronto, 40 St. George Street, Toronto, ON, Canada M5S 2E4}
\email{mbourque@math.toronto.edu}

\author[S. Lee]{Seewoo Lee}
\address{Department of Mathematics, Pohang University of Science and Technology, 77 Cheongam-Ro, Nam-Gu, Pohang, Gyeongbuk, Korea 37673}
\email{seewoo5@gmail.com}

\author[B. Yang]{Bowen Yang}
\address{Amherst College, 220 South Pleasant Street
Amherst, MA 01002, USA}
\email{byang18@amherst.edu}


\begin{abstract}
We study two $2$-dimensional Teichm\"uller spaces of surfaces with boundary and marked points, namely, the pentagon and the punctured triangle. We show that their geometry is quite diffe\-rent from Teichm\"uller spaces of closed surfaces. Indeed, both spaces are exhausted by regular convex geodesic polygons with a fixed number of sides, and their geodesics diverge at most linearly.
\end{abstract}

\maketitle

\section{Introduction}

Let $\Sigma$ be a connected, compact, oriented surface with (possibly empty) boundary and let $P \subset \Sigma$ be a finite (possibly empty) set of marked points. The Teichm\"uller space $\teich(\Sigma,P)$ is the set of equivalence classes of pairs $(X,f)$ where $X$ is a bordered Riemann surface and $f:\Sigma \to X$ is an orientation-preserving homeomorphism (sometimes called a marking). Two pairs $(X,f)$ and $(Y,g)$ are equivalent if there is a conformal diffeomorphism $h: X \to Y$ such that $g^{-1} \circ h \circ f$ is isotopic to the identity rel $P$. The Teichm\"uller metric on $\teich(\Sigma,P)$ (to be defined in Section \ref{sec:prelim}) is complete, uniquely geodesic, and homeomorphic to $\RR^d$ for some $d\geq 0$. The dimension of $\teich(\Sigma,P)$ is
$$
d = 6g -6 + 3b + 2n + m + \sigma 
$$
where $g$ is the genus of $\Sigma$, $b$ is the number of boundary components, $n$ is the number of interior marked points, $m$ is the number of boundary marked points, and $\sigma$ is the dimension of the space of biholomorphisms $X \to X$ isotopic to the identity rel $f(P)$ for any $[(X,f)]$ in $\teich(\Sigma,P)$. This parameter $\sigma$ is equal to 
\begin{itemize}
\item $6$ for the sphere;
\item $4$ for the sphere with $1$ marked point;
\item $3$ for the disk;
\item $2$ for the torus, the sphere with $2$ marked points, and the disk with $1$ boundary marked point;
\item $1$ for the annulus, the disk with $1$ interior marked point, and the disk with $2$ boun\-dary marked points;
\item $0$ otherwise. 
\end{itemize}
When $\sigma=0$, the Teichm\"uller space $\teich(\Sigma,P)$ coincides with the space of complete hyperbolic metrics with totally geodesic boundary on $\Sigma \setminus P$ up to isometries isotopic to the identity.

After the pioneering work of Teichm\"uller, most people working on the subject res\-tricted their attention to the case where the surface $\Sigma$ is closed. One reason for this choice is that theorems are often simpler to state and prove in this context. Another reason is that by doubling a Riemann surface across its boundary, one obtains a closed surface with a symmetry, and most results which are true for closed surfaces hold automatically for surfaces with boundary via this doubling trick.

However, we feel that Teichm\"uller spaces of surfaces with boundary should not be ignored. They exhibit phenomena which are fundamentally different from the closed surface case. Moreover, they embed isometrically inside Teichm\"uller spaces of closed surfaces via the doubling trick. Thus what happens in these spaces also happens in spaces of closed surfaces. Finally, they serve a pedagogical purpose: the low-dimensional Teichm\"uller spaces are fairly easy to understand and illustrate the general theory in a concrete way.

For surfaces of small topological complexity, the Teichm\"uller metric can be des\-cribed explicitly. This is the case when $(\Sigma,P)$ is:
\begin{enumerate}
\item the disk with at most 3 marked points on the boundary (and none in the interior);
\item the disk with 1 marked point in the interior and at most 1 on the boundary;
\item the sphere with at most 3 marked points;
\item the disk with 4 marked points on the boundary;
\item the disk with 1 marked point in the interior and 2 on the boundary;
\item the disk with 2 marked points in the interior;
\item the annulus with at most 1 marked point on the boundary;
\item the sphere with 4 marked points;
\item the torus with at most 1 marked point. 
\end{enumerate}
The Teichm\"uller space $\teich(\Sigma,P)$ is a single point in cases (1)--(3), is isometric to $\RR$ in cases (4)--(7), and is isometric to the hyperbolic plane $\HH^2$ with curvature $-4$ in cases (8) and (9). We would like to add two entries to this list where we understand the Teichm\"uller metric at least qualitatively, namely when $(\Sigma,P)$ is:
\begin{enumerate}
\item[(10)] the disk with 5 marked points on the boundary;
\item[(11)] the disk with 1 marked point in the interior and 3 on the boundary.
\end{enumerate}
We call these surfaces the pentagon and the punctured triangle respectively, and denote them $\pentagon$ and $\triangletimes$. Their Teichm\"uller spaces are $2$-dimensional, yet are quite different from the hyperbolic plane. Note that if $(\Sigma,P)$ is:
\begin{enumerate}
\item[(12)] the annulus with 2 marked points on the same boundary component,
\end{enumerate}
then $\teich(\Sigma,P)$ is isometric to $\teich(\pentagon)$ (see Subsection \ref{sbsec:covering}). Only two Teichm\"uller spaces of dimension at most $2$ are missing from this list, namely when $(\Sigma,P)$ is:
\begin{enumerate}
\item[(13)] the disk with 2 marked points in the interior and 1 on the boundary;
\item[(14)] the annulus with 1 marked point on each boundary components.
\end{enumerate}
The Teichm\"uller spaces for (13) and (14) are isometric to one another. We hope to return to them in later work.

Our results are as follows.

\begin{thm} \label{thm:exhaustionpent}
$\teich(\pentagon)$ is a nested union of convex, regular, geodesic pentagons.
\end{thm}

\begin{thm} \label{thm:exhaustiontri}
$\teich(\triangletimes)$ is a nested union of convex, regular, geodesic triangles.
\end{thm}

Note the immediate consequence:

\begin{cor}
The convex hull of any compact set in $\teich(\pentagon)$ or $\teich(\triangletimes)$ is compact.
\end{cor}
\begin{proof}
Let $C$ be a compact set in $\teich(\pentagon)$ or $\teich(\triangletimes)$. By the previous theorem, $C$ is contained in some compact convex polygon $P$. The (closed) convex hull of $C$, being contained in $P$, is therefore compact.
\end{proof}

Whether this property holds for Teichm\"uller spaces in general is an open question of Masur \cite{MasurSurvey}.

We use these exhaustions by polygons to estimate the rate of divergence between geodesics in $\teich(\pentagon)$ and $\teich(\triangletimes)$. In any metric space, the divergence between two distinct geodesic rays $\gamma_1$ and $\gamma_2$ with $\gamma_1(0)=\gamma_2(0)=p$ at distance $t$ is defined as the infimum of lengths of paths joining $\gamma_1(t)$ and $\gamma_2(t)$ outside the ball of radius $t$ around $p$. In Euclidean space the divergence is linear in $t$ while it is exponential in hyperpolic space. Teichm\"uller spaces of closed surfaces are in some sense hybrids between Euclidean spaces and hyperbolic spaces since they contain quasi-isometric copies of both \cite{Bowditch} \cite{LeiningerSchleimer}. In that vein, Duchin and Rafi proved in \cite{DuchinRafi} that the divergence between geodesic rays is at most quadratic (and can be quadratic) in Teichm\"uller spaces of closed surfaces with marked points, when the dimension is at least $4$. In contrast, we show that divergence is at most linear in $\teich(\pentagon)$ and $\teich(\triangletimes)$.

\begin{thm} \label{thm:divergence}
The rate of divergence between any two geodesic rays from the same point in $\teich(\pentagon)$ or $\teich(\triangletimes)$ is at most linear. 
\end{thm}

Finally, we observe that $\teich(\pentagon)$ and $\teich(\triangletimes)$ have the following universal property:

\begin{thm} \label{thm:universal}
$\teich(\pentagon)$ and $\teich(\triangletimes)$ both embed isometrically in $\teich(\hexagon)$, the Teichm\"uller space of the hexagon, which in turn embeds isometrically in the Teichm\"uller space $\teich(\Sigma_g)$ of any closed surface of genus $g \geq 2$ (without marked points).
\end{thm}

Unlike Teichm\"uller disks, two distinct totally geodesic planes arising from such isometric embeddings can intersect in more than one point, hence along a geodesic. This is explained in Section \ref{sec:universal}.

\begin{acknowledgements}
This research was conducted during the 2016 Fields Undergraduate Summer Research Program. The authors thank the Fields Institute for providing this opportunity. MFB was partially supported by a postdoctoral research scholarship from the Fonds de recherche du Qu\'ebec -- Nature et technologies.
\end{acknowledgements}

\section{Preliminaries} \label{sec:prelim}

We start by recalling standard definitions and results from Teichm\"uller theory in their most general form. We then specialize to the case of the pentagon and the punctured triangle where many of these notions become quite simple.

\subsection{Quasiconformal maps}

A \emph{$K$-quasiconformal diffeomorphism} between bordered Riemann surfaces is a diffeomorphism whose derivative at all points distorts oriented angles by a factor at most $K$, or equivalently sends circles to ellipses of eccentricity at most $K$ and preserves orien\-ta\-tion \cite{AhlforsLectures}. A \emph{$K$-quasiconformal homeomorphism} is a limit of a sequence of $K_n$-quasiconformal diffeomorphisms such that $\liminf K_n \leq K.$ 

\subsection{Teichm\"uller metric}

The Teichm\"uller distance on $\teich(\Sigma,P)$ is defined as
$$
d([(X,f)],[(Y,g)]) = \inf \frac{1}{2} \log K 
$$
where the infimum is taken over all $K \geq 1$ such that there exists a $K$-quasiconformal homeomorphism $h:X \to Y$ with $g^{-1}\circ h \circ f$ isotopic to the identity rel $P$. 

From now on, we will suppress the marking $f: \Sigma \to X$ from our notation. All we need to remember is that any pair $X,Y \in \teich(\Sigma,P)$ comes with an isotopy class of homeomorphism $X \to Y$ rel the marked points.

\subsection{Quadratic differentials}
A \emph{quadratic differential} on $X \in \teich(\Sigma,P)$ is a tensor $q$ which takes the form $Q(z)dz^2$ in local coordinates for some function $Q$ which is holomorphic except possibly at the marked points, where it is allowed to have simple poles. Near a boundary point, if we take a coordinate chart which sends the boundary to the real axis, then it is required that the function $Q$ be real along the real axis. In other words, when evaluated at vectors tangent to the boundary of $X$, the tensor $q$ must return a value in $\RR\cup \{\infty\}$. 

Away from the singularities of $q$, the holomorphic $1$-form $\sqrt q$ can be integrated along arcs. On small enough simply-connected open sets this defines charts to $\CC$, called natural coordinates, in which $q$ becomes $dz^2$ \cite{Strebel}. These can be used to decompose $X$ into a union of Euclidean polygons with some sides identified via translations or central symmetries. The polygons can actually be chosen to be rectangles with sides parallel to the coordinate axes \cite[p.213]{Hubbard}, in which case we call the decomposition a rectangular structure.

\subsection{Teichm\"uller's theorem}

Teichm\"uller's theorem states that for any $X,Y \in \teich(\Sigma,P)$ with $X \neq Y$, the Teichm\"uller distance $d(X,Y)$ is equal to $\frac{1}{2}\log K$ for some $K$-quasicon\-for\-mal homeomorphism $h:X\to Y$ in the correct homotopy class. Moreover, there exist quadratic differentials on $X$ and $Y$ with respect to which $h$ has derivative $$\pm \left(\begin{smallmatrix} \sqrt{K} & 0 \\ 0 & 1/\sqrt{K} \end{smallmatrix}\right)$$ in natural coordinates away from singularities.

Conversely, a quasiconformal homeomorphism $h$ of the above form (called a Teichm\"uller homeomorphism) has minimal quasiconformal constant $K$ in its homotopy class. Furthermore, any $K$-quasiconformal homeomorphism $g$ homotopic to $h$ is equal to $h$ unless $\Sigma$ is an annulus or a torus and $P$ is empty, in which case $g$ can be equal to $h$ post-composed with a biholomorphism of $Y$ homotopic to the identity \cite{Teichmuller} \cite{Bers}. 

As a consequence, $\teich(\Sigma, P)$ is uniquely geodesic and the geodesic rays from a point $X$ are in one-to-one correspondence with the quadratic differentials of unit area on $X$. Although this seems to suggest that quadratic differentials are the tangent vectors to Teichm\"uller space, they are really the cotangent vectors. Tangent vectors can be represented as ellipse fields, and there is a natural bilinear pairing between tangent and cotangent vectors.

\subsection{Covering constructions} \label{sbsec:covering}

Let $f: (\Sigma, P) \to (\Pi, Q)$ be an orbifold covering. This means that for every $p \in \Sigma$, there are neighborhoods $U \ni p$ and $V \ni f(p)$, and embeddings $\varphi : U \to \RR^2$ and $\psi : V \to \RR^2$ such that $\psi \circ f \circ \varphi^{-1}$ is the restriction of a quotient map $\RR^2 \to \RR^2 / G$ where $G$ is a finite subgroup of $O(2)$.  The pullback map $\sigma_f : \teich(\Pi, Q) \to \teich(\Sigma, P)$ associates to any complex structure $\tau$ on $\Pi$ a complex structure $\sigma_f(\tau)$ on $\Sigma$ in such a way that $f$ is holomorphic or anti-holomorphic away from orbifold points with respect to those structures. 

A \emph{critical point} of $f$ is a point $c \in \Sigma$ such that $f$ is not injective in any neighborhood of $c$ with the following exception: if $c \in \Sigma^\circ$, $f(c) \in \partial \Pi$, and $f$ is $2$-to-$1$ in a neighborhood of $c$, then $c$ is not a critical point. In other words, interior points where $f$ acts as the quotient by a single reflection are not critical points. The set of critical points of $f$ is denoted $\crit(f)$.

The following result is folklore \cite[Section 6]{MMW}. The special case where the covering is assumed to be normal goes back to Teichm\"uller's original paper \cite[Section 28]{Teichmuller}.

\begin{thm} \label{thm:covering}
If $f: (\Sigma, P) \to (\Pi, Q)$ is an orbifold covering such that $$f^{-1}(Q) = P \cup \crit(f),$$ then the pullback map $\sigma_f$ is an isometric embedding.
\end{thm}
\begin{proof}
The condition $f^{-1}(Q) = P \cup \crit(f)$ implies that the lift of a Teichm\"uller homeomorphism by $f$ is again a Teichm\"uller homeomorphism. Indeed, simple poles of quadratic differentials pullback to either simple poles at marked points or to singularities of order $\geq 0$ at critical points. Since the quasiconformal dilatation of the Teichm\"uller homeomorphism upstairs is the same as the one downstairs, distance is preserved.
\end{proof}

An isometric embedding of Teichm\"uller spaces arising in this way is known as a covering construction. For example, there are orbifold coverings of degree $2$ from:
\begin{itemize}
\item the quadrilateral to the once-punctured bigon;
\item the annulus to the quadrilateral;
\item the annulus to the twice-punctured disk;
\item the torus to the four-times-punctured sphere;
\item the annulus with 2 marked points on the same boundary component to the pentagon;
\item the annulus with 1 marked point on each boundary component to the twice-punctured monogon.
\end{itemize}
All of these give rise to (surjective) isometries since the corresponding Teichm\"uller spaces have the same dimension.

\begin{figure}[htp]
\includegraphics[scale=.6]{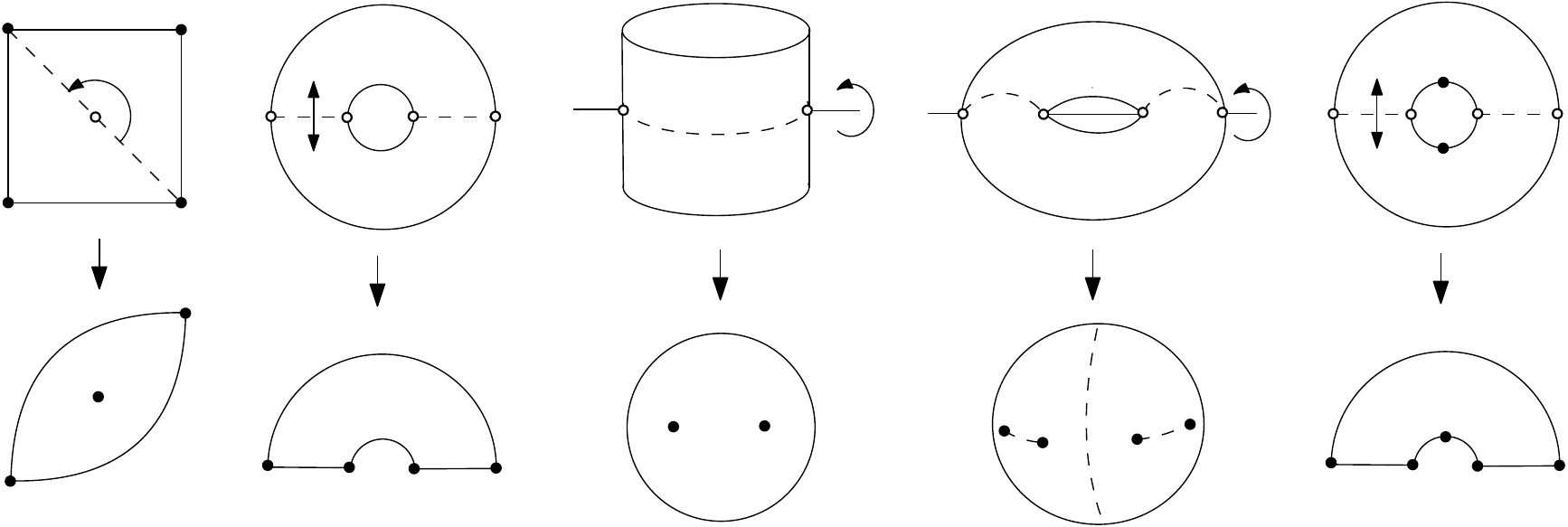}
\caption{Degree two orbifold coverings giving rise to isometries. The marked points are indicated in black and the critical points in white.}
\end{figure}

Another classical example comes from doubling. Given a surface $S=(\Pi, Q)$ with nonempty boundary, its double $R=(\Sigma,P)$ is the union of two copies of $S$, one with each possible orientation, with the boundaries glued via the identity map. The double $R$ comes with an orientation-reversing involution exchanging the two copies of $S$. The quotient by that involution is an orbifold covering $f:R \to S$ without critical points. Thus the Teichm\"uller space of any surface with boundary embeds isometrically in the Teichm\"uller space of some closed surface. The pullback map $\sigma_f:\teich(S) \to \teich(R)$ is simply the doubling construction, but done in the category of bordered Riemann surfaces. If $S$ has genus $g$, $b$ boundary components, $n$ interior marked points, and $m$ boundary marked points, then $R$ has genus $2g+b-1$ and $2n+m$ marked points. Assuming $S$ has negative Euler characteristic, then
$$
\dim \teich(R) = 6(2g+b-1)-6+2(2n+m) = 2(6g-6+3b+2n+m)= 2\dim \teich(S).
$$ 
The same equation holds when $S$ has non-negative Euler characteristic.

Theorem \ref{thm:universal} from the introduction is an easy consequence of Theorem \ref{thm:covering}: one only has to find appropriate orbifold coverings between the corresponding surfaces. The details are provided in Section \ref{sec:universal}.

\subsection{Measured foliations}

A \emph{measured foliation} on a compact surface $(\Sigma,P)$ is a foliation with isolated prong singularities ($1$-prong singularities are only allowed at the marked points) equipped with an invariant transverse measure \cite[p.56]{FLP}. The latter quantifies ``how many'' leaves of the foliation are crossing any given transverse arc. For example, if $q$ is a quadratic differential then its horizontal trajectories (maximal arcs along which $q>0$) form a measured foliation with transverse measure $|\im \sqrt{q}|$.

A \emph{multiarc} on $(\Sigma,P)$ is an embedded 1-dimensional submanifold $\alpha$ of $\Sigma \setminus P$ with boundary in $\partial \Sigma \setminus P$ such that
\begin{itemize}
\item no circle component of $\alpha$ bounds a disk or a once-punctured disk in $\Sigma \setminus P$;
\item no arc component of $\alpha$ bounds a disk with only $0$ or $1$ marked point on $\partial \Sigma$;
\item no two components of $\alpha$ are isotopic to each other in $\Sigma$ rel $P$.
\end{itemize}

The first two conditions define what it means for a simple closed curve or simple arc to be \emph{essential}. A \emph{weighted multiarc} is a multiarc together with a positive weight associated with each of its components. We generally consider (weighted) multiarcs only up to isotopy rel $P$. When we want to emphasize that we are talking about the isotopy class as opposed to a specific representative, we will write $[\alpha]$ for the isotopy class of $\alpha$.

Two measured foliations $F$ and $G$ are \emph{equivalent} if $i(\alpha,F)=i(\alpha,G)$ for every connected multiarc $\alpha$, where $i(\cdot,\cdot)$ is the geometric intersection number. The space $\mf(\Sigma,P)$ of equivalence classes of measured foliations on $(\Sigma,P)$ is given the weak topology by considering each measured foliation $F$ as a function on connected multiarcs via $\alpha \mapsto i(\alpha,F)$.  Every weighted multiarc $\beta$ can be enlarged to a measured foliation $F_\beta$ on $(\Sigma,P)$ such that $i(\alpha, \beta) = i(\alpha, F_\beta)$ for every connected multiarc $\alpha$. Thus the space of weighted multiarcs embeds inside the space of measured foliations.

For any $X \in \teich(\Sigma,P)$ and $F \in \mf(\Sigma,P)$, there exists a unique quadratic differential $q_F$ on $X$ whose horizontal foliation is equivalent to $F$. Moreover, the map $F \mapsto q_F$ is a homeomorphism. This is called the Hubbard--Masur (or heights) theorem \cite{HubbardMasur}. If $F$ is a weighted multiarc, then $q_F$ is called a \emph{Jenkins--Strebel differential}. 

The space of projective measured foliations $\pmf(\Sigma,P)$ is defined as the quotient of $\mf(\Sigma,P) \setminus \{0\}$ by positive rescaling. We will write $\overline{F}$ for the projective class of a measured foliation $F$. It follows from the Hubbard--Masur theorem that $\mf(\Sigma,P)$ is homeomorphic to $\RR^d$ and $\pmf(\Sigma,P)$ is homeo\-morphic to $\Sp^{d-1}$ where $d$ is the dimension of $\teich(\Sigma,P)$. 

\subsection{Extremal length}

There are three equivalent definitions of extremal length for a wei\-gh\-ted multiarc $\alpha = \sum h_i \cdot \alpha_i$ on a bordered Riemann surface $X \in \teich(\Sigma,P)$. The first one is
\begin{equation}
\el(\alpha, X) = \sup_\rho \frac{(\ell_\rho\,[\alpha])^2}{\area(\rho)}
\end{equation} 
where the supremum is over all Borel-measurable conformal metrics $\rho$ on $X$ and 
$$
\ell_\rho\,[\alpha] = \inf_{\gamma \in [\alpha]}  \sum_i h_i \int_{\gamma_i} \rho
$$
is the minimal weighted length of any rectifiable representative $\gamma = \sum h_i \cdot \gamma_i$ of $\alpha$. 

For example, the extremal length across a Euclidean rectangle is equal to its length divided by its height, and the extremal length around a Euclidean cylinder is its circumference divided by its height \cite[p.10]{AhlforsLectures}. Taking this as the definition of the extremal length of a rectangle or cylinder, the second definition of extremal length of a weighted multiarc is
\begin{equation}
\el(\alpha,X) = \inf \sum_i h_i^2  \el(R_i)
\end{equation}
where the infimum is taken over all collections $\{R_i\}$ of rectangles and cylinders embedded conformally and disjointly in $X$ with $R_i$ homotopic to $\alpha_i$.

The third definition of extremal length is 
\begin{equation}
\el(\alpha,X)=\area(q_\alpha)
\end{equation}
 where $q_\alpha$ is the quadratic differential on $X$ whose horizontal foliation is equivalent to $\alpha$. This definition extends to all measured foliations in view of the Hubbard--Masur theorem. 
 
The equivalence between the three definitions is explained in \cite{Dylan}. See also \cite{Ahlfors} for examples, properties and applications of extremal length.
 
\subsection{Kerckhoff's formula} 

Teichm\"uller distance can be expressed in terms of extremal lengths via Kerckhoff's formula \cite{Kerckhoff}:
\begin{equation} \label{eqn:kerck}
d(X,Y) = \sup_{\overline{F} \in \pmf(\Sigma,P)} \frac{1}{2}\log \frac{\el(F,Y)}{\el(F,X)}.
\end{equation} 
Moreover, the supremum is realized precisely when $F$ is the horizontal foliation of the initial quadratic differential for the Teichm\"uller homeomorphism $X \to Y$. Note that the supremand in (\ref{eqn:kerck}) does not depend on the choice of $F \in \overline{F}$. Indeed, extremal length scales quadratically in the sense that $\el(\lambda F, X) = \lambda^2 \el(F,X)$ for every $F \in \mf(\Sigma,P)$ and $\lambda > 0$. 

\section{Pentagons}

\subsection{Representation}
An element of $\teich(\pentagon)$ is (an equivalence class of) a bordered Riemann surface $X$ homeo\-morphic to the closed disk together with a 5-tuple $\vec{x}=(x_1,x_2,...,x_5)$ of distinct points appearing in counter-clockwise order along $\partial X$. Two pairs $(X,\vec{x})$ and $(Y,\vec{y})$ are equivalent if there is a conformal diffeomorphism $h:X \to Y$ such that $h(x_j)=y_j$ for $j=1,...,5$. We don't need a marking from a base topological surface here, since the labelling of the marked points provides the same information. For convenience, the index $j$ will be taken modulo $5$ so that $5+1=1$ and $1-1=5$.

By the Riemann Mapping Theorem, every element of $\teich(\pentagon)$ can be represented uniquely as the closed upper half-plane $\Hbar$ with 5-tuple $(x_1,x_2,\infty,-1,0)$, where $0<x_1<x_2$. In particular, we see that $\teich(\pentagon)$ is homeomorphic to $\RR^2$ via the coordinates 
$$
\left(\Hbar,(x_1,x_2,\infty,-1,0)\right) \mapsto (\log(x_1),\log(x_2-x_1)).
$$ 
One could also represent elements of $\teich(\pentagon)$ with the closed unit disk, but we found the upper half-plane to be more convenient.

From the point of view of hyperbolic geometry, $\teich(\pentagon)$ is the space of ideal pentagons in $\HH^2$ with labelled vertices up to isometry, or the space of right-angled pentagons with labelled vertices up to isometry. There are other equivalent definitions. For example, $\teich(\pentagon)$ is the space of Euclidean pentagons with 5 prescribed angles up to similarity. 

\subsection{The five axes of symmetry}
The dihedral group $D_5$ acts on $\teich(\pentagon)$ by permuting the labels of the marked points and reversing orientation when the permutation does so. This action is isometric with respect to the Teichm\"uller metric. There are $5$ special geodesics in $\teich(\pentagon)$ given by the loci of fixed points of the 5 reflections in $D_5$. For example, the permutation $(25)(34)$ fixes all pentagons $(X,\vec{x})$ which admit an anti-conformal involution $h$ such that $h(x_1)=x_1$, $h(x_2)=x_5$ and $h(x_3)=x_4$. This locus is a geodesic. Indeed, the quotient of $\pentagon$ by any of these reflections is an orbifold covering onto a quadrilateral. Hence it gives rise to an isometric embedding of the Teichm\"uller space of quadrilaterals into $\teich(\pentagon)$. But the Teichm\"uller space of quadrilaterals is isometric to the real line by Gr\"otzsch's theorem (a special case of Teichm\"uller's theorem).  By definition, a geodesic is an isometric embedding of the real line.

Let us denote by $\sigma_j$ the reflection in $D_5$ which fixes the vertex labelled $j$. If $(X,\vec{x})$ is realized as the upper half-plane with marked points $(x_1,x_2,\infty,-1,0)$, then the locus $\gamma_1=\fix(\sigma_1)$ is given by the equation $ x_2+1 = (x_1+1)^2$. The reason for this is that every anti-conformal involution of $\Hbar$ is either an inversion in a circle centered on the real line or a reflection in a vertical line. Now, the the anti-conformal involution realizing the permutation $\sigma_1$ on $(x_1,x_2,\infty,-1,0)$ must fix $x_1$, swap $x_2$ and $0$, and swap $\infty$ and $-1$. The involution is therefore equal to the inversion in the circle centered at $-1$ passing through $x_1$. The above equation is just the condition that $|x_2 - (-1)||0 - (-1)| = |x_1 - (-1)|^2$. Similarly, 
\begin{itemize}
\item $\gamma_2=\fix(\sigma_2)$ has equation $x_1(x_1+1) =(x_2-x_1)^2$;
\item $\gamma_3=\fix(\sigma_3)$ has equation $x_2 = x_1+1$;
\item $\gamma_4=\fix(\sigma_4)$ has equation $x_1 x_2 = 1$ subject to $x_1<1$;
\item $\gamma_5=\fix(\sigma_5)$ has equation $(x_2-x_1)(x_2+1) = x_2^2$.
\end{itemize}

Let $\phi = \left(1+\sqrt{5}\right)/2$ be the golden ratio. We leave it to the reader to check that $x_1 = 1/\phi$ and $x_2 = \phi$ satisfy all of the above equations. In other words, the regular pentagon (which is fixed by all of $D_5$) is conformally equivalent to the upper half-plane with marked points $(1/\phi,\phi,\infty, -1, 0)$. We call this point the \emph{origin} of $\teich(\pentagon)$. The geodesics $\gamma_j$ all intersect at the origin and this is the only intersection point of any two of them. 

\begin{figure}[htp]
\includegraphics[scale=.7]{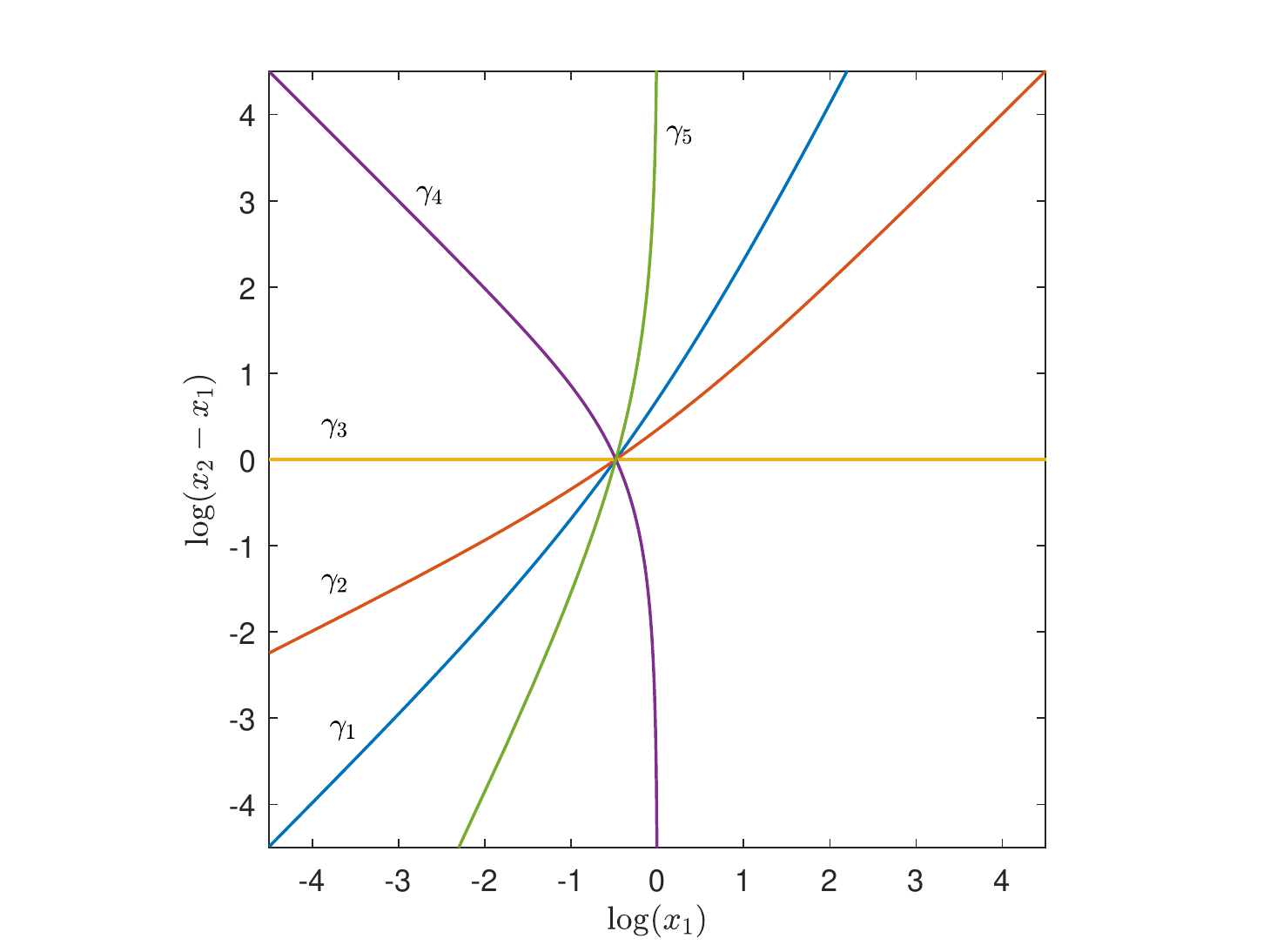}
\caption{The five axes of symmetry of $\teich(\pentagon)$ plotted in $\log$-coordinates.}
\end{figure}

\subsection{Measured foliations}

Measured foliations on the pentagon are of the simplest possible kind.

\begin{lem} \label{lem:multiarc}
Every measured foliation of $\pentagon$ is a weighted multiarc.
\end{lem}
\begin{proof} Let $F$ be a measured foliation on $\pentagon$. It suffices to prove that every leaf of $F$ is a proper arc. Suppose not, i.e., let $\lambda$ be a leaf of $F$ which is recurrent to some part of $\pentagon$. Let $\alpha$ be a short arc transverse to $F$ to which $\lambda$ returns. Starting from $\alpha$, follow $\lambda$ until it first returns to $\alpha$. The region enclosed by these arcs is a disk. Doubling this disk across the boundary, we get a measured foliation on the sphere with at most two 1-prong singularities (where $\alpha$ and $\lambda$ meet). But a measured foliation on the sphere must have at least four 1-prong singularities by the Euler--Poincar\'e formula \cite[p.58]{FLP}.  
\end{proof}

A multiarc on $\pentagon$ can have either $1$ or $2$ components. Thus the space $\pmf(\pentagon)$ has the structure of a graph whose vertices correspond to essential arcs and whose edges correspond to pairs of disjoint essential arcs (the position of a point along an edge indicates the relative weights on the corresponding arcs). Since there are 5 essential arcs in $\pentagon$ and each arc is disjoint from exactly two other essential arcs, $\pmf(\pentagon)$ is isomorphic to a pentagon. We use the following notation for the essential arcs in $\pentagon$. For each $j \in \{1,2,3,4,5\}$, the arc $\alpha_j$ is the one which separates the vertex labelled $j$ and its two neighbors in $\partial\pentagon$ from the other two vertices (see Figure \ref{fig:pmf}). Equivalently, $\alpha_j$ is the isotopy class of essential arc which is sent to itself by $\sigma_j$.

\begin{figure}[htp] 
\includegraphics[scale=.8]{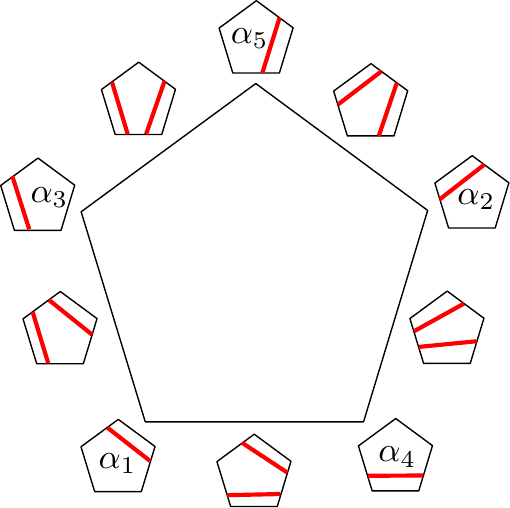}
\caption{$\pmf(\pentagon) \cong \pentagon$. In each small pentagon, the bottom left corner is the vertex labelled $1$ and the remaining vertices are labelled in counterclockwise order.} \label{fig:pmf}
\end{figure}

\subsection{Quadratic differentials}

Similarly, quadratic differentials and the rectangular structures they induce on the pentagon are easy to describe geometrically.

\begin{lem} \label{lem:lshape}
Every rectangular structure on $\pentagon$ is a (possibly degenerate) $L$-shape.
\end{lem}
\begin{proof}
Let $q$ be a quadratic differential on $\Hbar$ with marked points at $x_1$, $x_2$, $\infty$, $-1$ and $0$. Recall that $q$ has at most simple poles at the marked points. Since $q$ is real along $\RR$, it extends to a quadratic differential on $\CHAT$ which is symmetric about the real axis. By the Euler--Poincar\'e formula (or by considering the quadratic differential $dz^2$, which has a pole of order $4$ at infinity), the degree of the divisor of $q$ is $-4$. 

If $q$ has exactly 4 simple poles, then it has no other singularities and the corresponding rectangular structure is a rectangle. This is because the sign of $q$ along $\RR$ changes exactly at the poles, so the image of $\Hbar$ under the natural coordinate for $q$ is a polygon with 4 sides which are alternatingly horizontal and vertical. Note that the rectangle has one marked point along one of its sides. We call this a degenerate $L$-shape. 

Otherwise, $q$ has a simple pole at each of the $5$ marked points as well as $1$ simple zero. Since the zeros of $q$ are symmetric about the real axis, its only zero must be on the real line. Therefore the natural coordinate $z \mapsto \int_i^z \sqrt{q}$ is globally defined on $\Hbar$. Its image is an immersed polygon with sides parallel to the axes, $5$ corners of angle $\pi/2$ (corresponding to the poles) and $1$ corner of angle $3\pi /2$ (corresponding to the zero). Any such polygon is actually embedded, and looks like the letter $L$ up to reflections in the coordinate axes.  
\end{proof}

\begin{figure}[htp] 
\includegraphics[scale=.5]{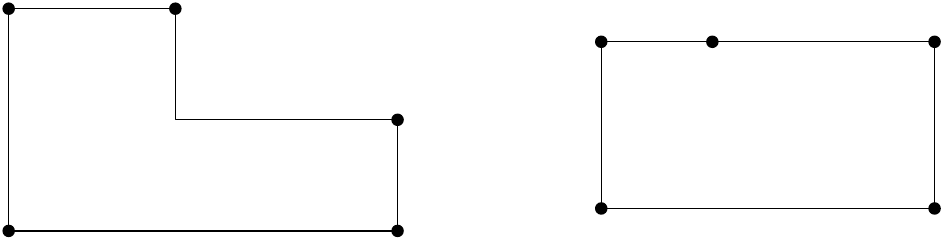}
\caption{An $L$-shape and a degenerate $L$-shape.} \label{fig:lshapes}
\end{figure}

\subsection{Parametrizing the axes}

We parametrize each of the 5 geodesics $\gamma_j$ by arclength with $\gamma_j(0)$ equal to the origin. It remains to orient them. Since $\gamma_j$ is fixed pointwise by the reflection $\sigma_j$, the horizontal and vertical foliations for its defining quadratic differential are also fixed by $\sigma_j$. Up to scaling, there are only two measured foliations invariant by $\sigma_j$, namely $\alpha_j$ and $\alpha_{j-1}+\alpha_{j+1}$. We orient $\gamma_j$ by declaring that $\alpha_{j-1}+\alpha_{j+1}$ is the horizontal foliation and  $\alpha_j$ is the vertical foliation for the quadratic differential. This way, $\alpha_j$ gets pinched along $\gamma_j$ in the sense that $\el(\alpha_j,\gamma_j(t)) \to 0$ as $t \to +\infty$.

The origin splits the $5$ geodesics $\gamma_j$ into 10 rays $\gamma_j^\pm$, and their order of appearance around the origin is the same as the order of appearance of their vertical foliation in $\pmf(\pentagon)$. This implies that $\gamma_1^+$ is followed by $\gamma_2^-$, then $\gamma_3^+$, and so on (see Figure \ref{fig:rays}). In other words, the geodesics appear in sequential order around the origin but with alternating orientation. 

\begin{figure}[htp] 
\includegraphics[scale=.9]{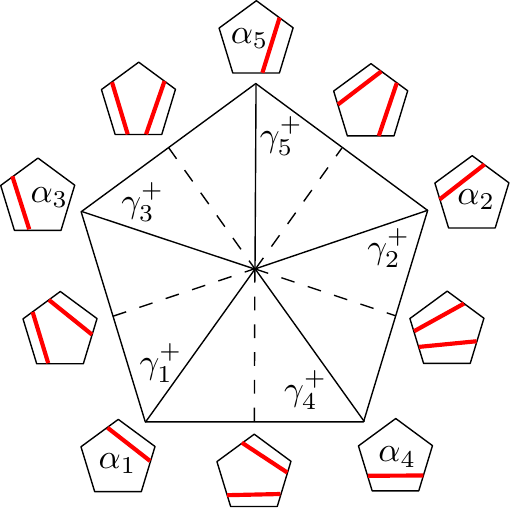}
\caption{The cyclic order and orientation of the axes of symmetry in $\teich(\pentagon)$. The backward direction of each axis is dashed. $\pmf(\pentagon)$ is drawn to indicate the projective class of the vertical foliation for each ray. This is intended only as a visual guide; it does not correspond to a compactification of $\teich(\pentagon)$.} \label{fig:rays}
\end{figure}

\subsection{Half-planes}

We define an \emph{open half-plane}  in $\teich(\pentagon)$ to be either connected components of the complement of a geodesic. A \emph{closed half-plane} is the closure of an open half-plane, i.e., an open half-plane together with its defining geodesic. 

\begin{lem}
Closed half-planes are convex.
\end{lem}

\begin{proof}
Suppose that a closed half-plane $H$ is not convex. Then there is a geodesic segment $[x,y]$ with endpoints in $H$ which is not contained in $H$. Consider a maximal subinterval $(z,w)\subset [x,y]$ which is contained in the complement of $H$. Then $z$ and $w$ belong to $\partial H$ by maximality. Since $\partial H$ is a geodesic and the geodesic between any two points is unique, the segment $[z,w]$ is contained in $\partial H \subset H$, which is a contradiction. 
\end{proof}

\subsection{Pentagons in the space of pentagons}
For any $t>0$, we define $P_t$ to be the geodesic pentagon with vertices $\gamma_1(t), \gamma_3(t), \gamma_5(t), \gamma_2(t), \gamma_4(t)$ together with the region it bounds. More precisely,
$$ P_t  = \bigcap_{j=1}^5 H^j(t)$$ 
where $H^j(t)$ is the closed half-plane bounded by the geodesic through $\gamma_j(t)$ and $\gamma_{j+2}(t)$ which contains the origin. 

\begin{lem}
$P_t$ is convex for any $t>0$.
\end{lem}
\begin{proof}
$P_t$ is the intersection of $5$ closed half-planes each of which is convex. 
\end{proof}

\begin{lem}
If $0< s < t$, then $P_s \subset P_t$.
\end{lem}
\begin{proof}
First observe that the vertices of $P_s$ are contained in $P_t$. Since $P_s$ is the convex hull of its vertices and $P_t$ is convex, the inclusion follows.
\end{proof}

By construction, $P_t$ is also regular since $D_5$ acts on it by isometries in a faithful manner. The only part of Theorem \ref{thm:exhaustionpent} left to prove is that $\teich(\pentagon)=\bigcup_{t>0} P_t$.

\subsection{Symmetric geodesics}

In order to prove that the pentagons $P_t$ exhaust $\teich(\pentagon)$, we will shift our point of view slightly. We need to better understand the geodesics that form the sides of $P_t$. What can we say about the geodesic through $\gamma_2(t)$ and $\gamma_5(t)$ for example? What do the underlying rectangular structures look like? To answer this, observe that the isometry of $\teich(\pentagon)$ induced by the permutation $\sigma_1$ switches the points $\gamma_2(t)$ and $\gamma_5(t)$. Therefore it sends the geodesic through $\gamma_2(t)$ and $\gamma_5(t)$ to itself in an orientation-reversing manner, thereby fixing the midpoint of the segment $[\gamma_2(t),\gamma_5(t)]$. 

We will say that a geodesic which is sent to itself in an orientation-reversing manner by $\sigma_1$ is \emph{symmetric about $\gamma_1$}. It is interesting to note that the geodesics symmetric about $\gamma_1$ foliate $\teich(\pentagon)$. This is ana\-lo\-gous to the existence and uniqueness of perpendiculars in the Euclidean plane and the hyperbolic plane. 

\begin{lem} \label{lem:perp}
For any $x\in \teich(\pentagon)$, there exists a unique geodesic through $x$ which is symmetric about $\gamma_1$.
\end{lem}
\begin{proof}
First assume that $x$ does not belong to the axis of reflection $\gamma_1$. Then $\sigma_1(x) \neq x$ and the geodesic through these two points is sent to itself in an orientation-reversing manner by $\sigma_1$. Conversely, if $\eta$ is a geodesic containing $x$ and $\sigma_1(\eta) = \eta$, then $\eta$ contains $\sigma_1(x)$, which proves uniqueness.

Now suppose that $x \in \gamma_1$. Consider a non-zero tangent vector $v$ to $\gamma_1$ at $x$. The space of quadratic differentials $q$ on $x$ which pair trivially with $v$ is $1$-dimensional. Let $q\neq 0$ be such a quadratic differential. Since $\sigma_1$ fixes $v$ and preserves the pairing between tangent and cotangent vectors, it sends $q$ to a quadratic differential of the same norm which pairs trivially with $v$ yet is different from $q$, i.e., to $-q$. Thus $\sigma_1$ sends the geodesic cotangent to $q$ to the geodesic cotangent to $-q$, that is, to itself in an orientation-reversing manner. 

Conversely, let $\eta$ be a geodesic through $x$ which is symmetric about $\gamma_1$ and let $q$ be its unit cotangent vector at $x$. Then  $\sigma_1$ sends $q$ to $-q$ while it fixes $v$. Since $\sigma_1$ is an isometry, it preserves the pairing between tangent and cotangent vectors, so that
$$
\langle v , q\rangle = \langle v , -q\rangle \quad \Rightarrow \quad  \langle v , q\rangle = 0.
$$
As we observed before, the orthogonal complement $v^\perp$ is $1$-dimensional, which means that $q$ is determined up to a scalar and that $\eta$ is unique.
\end{proof}

Actually, the geodesics symmetric about $\gamma_1$ can be described explicitly. For any $a>0$, consider the $L$-shape $\Phi_a$ with vertices at $0$, $(1+a)$, $(1+a)+i$, $1+i$, $1+(1+a)i$ and $(1+a)i$ where all vertices except $1+i$ are marked and the first marked point is the origin (see Figure \ref{fig:symL}). Let $R$ be the reflection about the line $y=x$. Observe that $R(\Phi_a) = \Phi_a$ and that $R$ acts as the permutation $\sigma_1 = (25)(34)$ on the marked points. Thus $\Phi_a$ represents a point on $\gamma_1$. More generally, for any $t \in \RR$ we have
$$
R \left( \begin{pmatrix} e^t & 0 \\ 0 & e^{-t} \end{pmatrix} \cdot \Phi_a \right) = \begin{pmatrix} e^{-t} & 0 \\ 0 & e^{t} \end{pmatrix} \cdot \Phi_a
$$ 
meaning that Teichm\"uller flow followed by reflection is the same as negative Teichm\"uller flow. In particular, the Teichm\"uller geodesic $\eta_a = \{ \g_t \Phi_a \mid t \in \RR \}$ cotangent to $\Phi_a$ is sent to itself in an orien\-ta\-tion-reversing manner by $\sigma_1$. 

\begin{figure}[htp] 
\includegraphics[scale=.8]{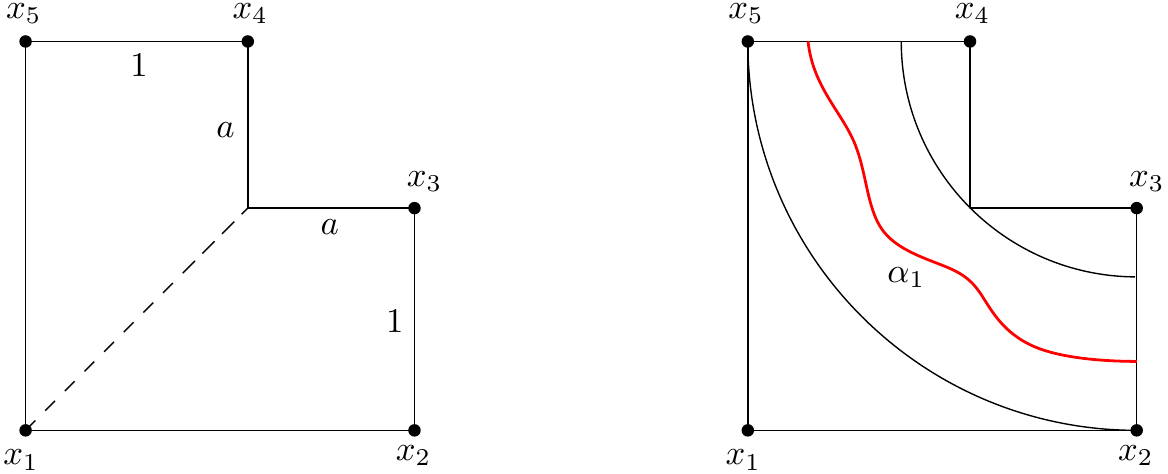}
\caption{The symmetric $L$-shape $\Phi_a$ and an embedded circular rectangle homotopic to $\alpha_1$.} \label{fig:symL}
\end{figure}

\begin{remark}
The geodesic $\eta_{1/4}$ was used in \cite{FBRafi} to prove the exis\-ten\-ce of a non-convex ball in $\teich(\pentagon)$. The proof presented there implies that some ball $B$ centered on $\gamma_1$ is such that a segment of $\eta_{1/4}$ symmetric about $\gamma_1$ has its endpoints in $B$ but its midpoint $\Phi_{1/4}$ outside $B$. However, the ball $B$ could have very large radius a priori. In the course of this project, we found numerical evidence suggesting that there is a non-convex ball of radius less than $1$. 
\end{remark}

We now show that every geodesic symmetric about $\gamma_1$ is of this form.

\begin{prop} \label{prop:explicit}
Any geodesic symmetric about $\gamma_1$ is equal to $\eta_a$ for a unique $a>0$.
\end{prop}
\begin{proof}
We already observed that $\eta_a$ is symmetric about $\gamma_1$ for any $a>0$. If $\tau$ is a geodesic symmetric about $\gamma_1$, then it intersects $\gamma_1$ at some point $x$. By uniqueness of the symmetric geodesic through $x$, it suffices to prove that $x \in \eta_a$ for a unique $a>0$. In other words, we have to show that the map $a \mapsto \Phi_a$ from $(0,\infty)$ to $\gamma_1$ is a bijection. 

 Observe that $\gamma_1(t)$ can be represented by a rectangle of length $ e^t / \sqrt{c_0}$ and height $\sqrt{c_0} e^{-t}$ with vertex $x_1$ in the middle of the left side, where $c_0 = \el(\alpha_1,\gamma_1(0))$. Indeed, this describes a Teichm\"uller geodesic fixed pointwise by $\sigma_1$. In particular, the map $\gamma_1(t) \mapsto \el(\alpha_1,\gamma_1(t)) = c_0 e^{-2t}$ is a bijection from $\gamma_1$ to $(0,\infty)$. Thus in order to prove the above statement, it suffices to show that the map
$$ 
a \mapsto \el(\alpha_1,\Phi_a)
$$
is a bijection of $(0,\infty)$ onto itself.

If $0<a<b$, then $\Phi_a \subset \Phi_b$. Let $q$ be the quadratic differential on $\Phi_a$ realizing the extremal length of $\alpha_1$ and let $\rho = \sqrt{|q|}$ be the corresponding conformal metric. We extend $\rho$ to a conformal metric $\wtilde{\rho}$ on $\Phi_b$ by setting it to be $0$ on $\Phi_b \setminus \Phi_a$. Every arc homotopic to $\alpha_1$ on $\Phi_b$ contains a subarc homotopic to $\alpha_1$ on $\Phi_a$ so that 
$$
\ell_{\wtilde \rho}\,[\alpha_1] = \ell_{\rho}\,[\alpha_1].$$ 
Clearly, $\wtilde{\rho}$ is not the extremal metric on $\Phi_b$ hence
$$
\el(\alpha_1, \Phi_b) > \frac{(\ell_{\wtilde{\rho}}\,[\alpha_1])^2}{\area(\wtilde{\rho})}= \frac{(\ell_{\rho}\,[\alpha_1])^2}{\area(\rho)} = \el(\alpha_1,\Phi_a).
$$
This shows that extremal length is monotone in $a$. 

It remains to prove surjectivity. For $0<a<\frac{1}{\sqrt{2}-1}$, the $L$-shape $\Phi_a$ contains a quarter of an annulus centered at $(1+a)+(1+a)i$ with inner radius $a\sqrt{2}$ and outer radius $(1+a)$ (see Figure \ref{fig:symL}). The extremal length around this circular strip is equal to 
$$
\frac{\pi /2}{\log(1+a) - \log(a \sqrt{2})}
$$ 
which is an upper bound for $\el(\alpha_1, \Phi_a)$. This implies that $\el(\alpha_1, \Phi_a) \to 0$ as $a \to 0$. On the other hand, the Euclidean metric $\rho$ on $\Phi_a$ gives the lower bound
$$
\el(\alpha_1, \Phi_a)  \geq \frac{(2a)^2}{1 + 2a}
$$
which tends to infinity with $a$. By continuity, every positive value is attained.
\end{proof}

\begin{figure}[htp] 
\includegraphics[scale=.9]{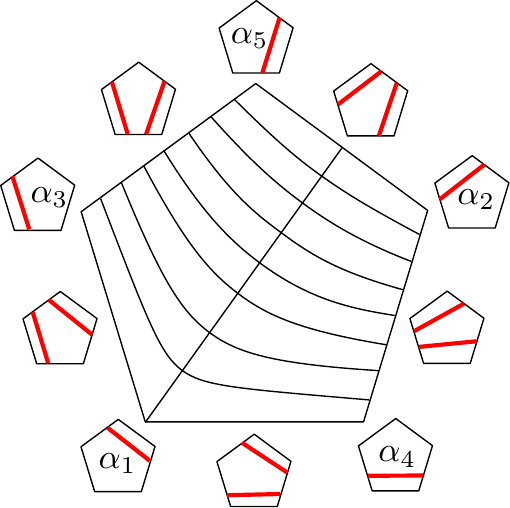}
\caption{$\teich(\pentagon)$ is foliated by geodesics $\eta_a$ symmetric about $\gamma_1$. The projective horizontal and vertical foliations for $\eta_a$ are $\overline{\alpha_4 + a \alpha_2}$ and $\overline{\alpha_3+a \alpha_5}$ respectively.} \label{fig:perp}
\end{figure}

Let $U_a$ be the closed half-plane bounded by $\eta_a$ which points towards $\gamma_1^+$. By Lemma \ref{lem:perp} and Proposition \ref{prop:explicit}, these half-planes exhaust $\teich(\pentagon)$ as $a \nearrow \infty$. Similarly, the sets
$$
Q_a = \bigcap_{j=1}^5 \sigma_j(U_a)
$$ 
exhaust $\teich(\pentagon)$ as $a \nearrow \infty$. This almost implies what we want. The issue here is that a priori $Q_a$ could be non-compact for large $a$, as would happen in the hyperbolic plane for example.  What we need to show is that each side of $Q_a$ intersects its neighbors and hence that $Q_a$ is equal to $P_t$ for some $t>0$, provided that $a$ is large enough so that $Q_a$ is not empty. Figure \ref{fig:perp} suggests the proof: the projective classes of the horizontal and vertical foliations for $\eta_a$ are linked with those of $\sigma_5(\eta_a)$ in $\pmf(\pentagon)$, forcing $\eta_a$ and $\sigma_5(\eta_a)$ to intersect. 

In order to make that argument rigorous, one needs to put a topology on $$\teich(\pentagon)\cup \pmf(\pentagon)$$ in which the closure of $\eta_a$ disconnects the endpoints of $\sigma_5(\eta_a)$. Thurston's compac\-ti\-fication \cite[p.118]{FLP} ---which is homeomorphic to a closed disc--- does the job. By Lemma \ref{lem:multiarc} every geodesic ray in $\teich(\pentagon)$ is Jenkins--Strebel, hence converges in Thurston's boundary to the vertex corresponding to its vertical foliation or to the center of the open edge containing its vertical foliation \cite{MasurTwoBoundaries}. In particular, the geodesics $\eta_a$ all converge to $\overline{\alpha_4 + \alpha_2}$ in the backward direction and to $\overline{\alpha_3 + \alpha_5}$ in the forward direction, while $\sigma_5(\eta_a)$ converges to $\overline{\alpha_1+\alpha_3}$ and $\overline{\alpha_2 + \alpha_5}$.

We will give another proof that $\eta_a$ intersects $\sigma_5(\eta_a)$ which yields more information such as estimates on the lengths of the sides of $Q_a$. Observe that $\eta_a$ intersects $\sigma_5(\eta_a)$ if and only if $\eta_a$ intersects $\gamma_5$, and this is what we will show. To do this, we will characterize $\gamma_5$ as the set of solutions to an equation involving extremal length and then use the intermediate value theorem.

\subsection{Equal extremal lengths implies symmetry}

Recall that $\alpha_5$ is the arc in $\pentagon$ which separates the vertices $4$, $5$, $1$ from $2$ and $3$. By conformal invariance of extremal length, if $X \in \gamma_5$ then
$$
\el(\alpha_1,X) = \el(\alpha_4,X)
$$
as $\sigma_5$ permutes the arcs $\alpha_1$ and $\alpha_4$. The converse is also true.

\begin{lem} \label{lem:equalissym}
Let $X \in \teich(\pentagon)$. Suppose that $\el(\alpha_1,X) = \el(\alpha_4,X)$. Then $X \in \gamma_5$, i.e., $X$ admits an anti-conformal involution fixing the vertex $x_5$.
\end{lem}
\begin{proof}
Map $X$ conformally onto a rectangle in such a way that the vertex $x_5$ is on a side and the other vertices are at the corners of the rectangle. Suppose that the segment $[x_4, x_5]$ is strictly shorter than $[x_5,x_1]$. Then the topological quadri\-la\-teral joining $[x_4, x_5]$ to $[x_2,x_3]$ embeds conformally in (and is different from) the quadrilateral joining $[x_5, x_1]$ to $[x_2,x_3]$. To see this, simply reflect about the perpendicular bisector of $[x_4,x_1]$. By monotonicity of extremal length, this implies that $\el(\alpha_1,X) > \el(\alpha_4,X)$ which is a contradiction. As the argument is symme\-tric in $x_1$ and $x_4$, the vertex $x_5$ must lie in the middle of its side. The reflection of the rectangle about the perpendicular bisector of $[x_4,x_1]$ is an anti-conformal involution of $X$ fixing $x_5$. 
\end{proof}

\subsection{Extremal length estimates}

By the previous subsection, $\gamma_5$ is the locus of points $X$ in $\teich(\pentagon)$ such that $\el(\alpha_1,X) = \el(\alpha_4,X)$. Recall also that $$\eta_a= \{ \g_t \Phi_a \mid t \in \RR \}$$ where $\g_t$ is the diagonal matrix $\left(\begin{smallmatrix} e^t & 0 \\ 0 & e^{-t} \end{smallmatrix}\right)$
and $\Phi_a$ is the symmetric $L$-shape with legs of length $a$. Note that $\g_t  \Phi_a$ is conformally equivalent to $\h_t  \Phi_a$ where $\h_t =\left(\begin{smallmatrix} e^{2t} & 0 \\ 0 & 1 \end{smallmatrix}\right).$ We will use this rescaling when convenient for calculations.

\begin{prop} \label{prop:intersection}
If $a \geq 2$, then $\eta_a$ intersects $\gamma_5$. More precisely, $\g_t\Phi_a$ belongs to $\gamma_5$ for some $t \in [0, \log(1+a)]$.
\end{prop}

We break down the proof into several lemmata. The main idea is that at $t=0$ we have $\el(\alpha_1, \g_t\Phi_a) \geq \el(\alpha_4, \g_t\Phi_a)$ while the inequality is reversed at $t=\log(1+a)$. By the intermediate value theorem, equality occurs somewhere in between.

\begin{lem}
For every $a>0$, we have 
$$
\el(\alpha_1, \Phi_a)  \geq \frac{4a^2}{1 + 2a}.
$$
\end{lem}
\begin{proof}
Use the first definition of extremal length with the Euclidean metric on $\Phi_a$ (see the proof of Proposition \ref{prop:explicit}).
\end{proof}

\begin{lem}
For every $a>0$, we have 
$$
\el(\alpha_4, \Phi_a)  \leq 1+a.
$$
\end{lem}
\begin{proof}
There is a horizontal rectangle of length $1+a$ and height $a$ embedded in the homotopy class of $\alpha_4$.
\end{proof}

\begin{cor} \label{cor:ineq}
If $a\geq \frac{3+\sqrt{17}}{4}$, then $\el(\alpha_1, \Phi_a) \geq \el(\alpha_4, \Phi_a)$.
\end{cor}
\begin{proof}
The condition implies  $$\frac{4a^2}{1 + 2a} \geq 1+a.$$ The conclusion follows from the previous lemmata. 
\end{proof}

\begin{lem} \label{lem:upperalpha1}
For every $a>0$ and $t>0$, we have 
$$
\el(\alpha_1, \g_t \Phi_a)  \leq 1+a+e^{2t}a.
$$
\end{lem}
\begin{proof}
Let $K = e^{2t}$. Let $\Gamma$ be the family of all essential arcs in $\pentagon$ which intersect every representative of $\alpha_1$. As a set we have $\Gamma = \alpha_2 \cup \alpha_5$. This should not be confused with $\alpha_2 + \alpha_5$: each element of $\Gamma$ is a single arc, not a multiarc. By duality of extremal length for rectangles, 
$$\el(\alpha_1, \g_t \Phi_a) = \frac{1}{ \el(\Gamma, \g_t \Phi_a)} = \frac{1}{ \el(\Gamma, \h_t \Phi_a)}. $$

Consider the metric $\rho$ which is defined to be $|dz|$ at points in $\h_t \Phi_a$ with real part bigger than $(K-1)$ and $0$ elsewhere. In other words, $\rho$ is the Euclidean metric on $\h_t \Phi_a$ but with a $(K-1)\times (1+a)$ rectangle cut off on the left. The distance across the leftover region (from the two upper-right sides to the two lower-left sides) is at least $1$, while its area is equal to $1+a+Ka$. This shows that

$$
\el(\Gamma, \h_t \Phi_a) \geq \frac{1}{1+a+Ka}
$$
from which the conclusion follows.
\end{proof}

\begin{figure}[htp] 
\includegraphics[scale=.8]{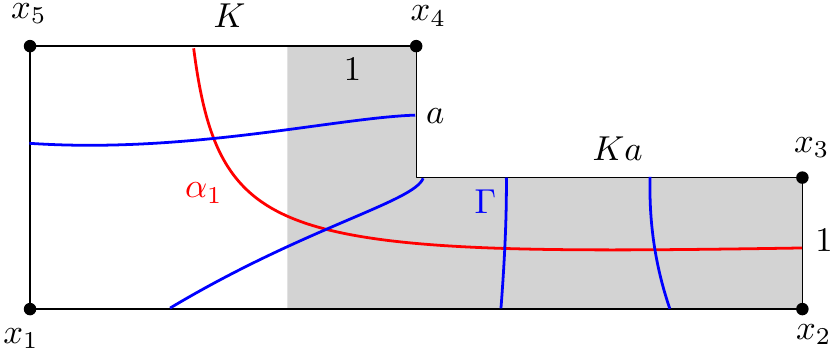}
\caption{The conformal metric $\rho$ in the proof of Lemma \ref{lem:upperalpha1} is equal to the Euclidean metric on the shaded region and zero elsewhere. Every arc in the family $\Gamma$ has length at least $1$ with respect to $\rho$. The extremal length of $\alpha_1$ is the reciprocal of the extremal length of $\Gamma$.}
\end{figure}

\begin{lem}
For every $a>0$ and $t>0$, we have 
$$
\el(\alpha_4, \g_t \Phi_a)  \geq e^{2t}\left(\frac{1}{1+a}+ a\right).
$$
\end{lem}
\begin{proof}
Let $K = e^{2t}$. Consider the metric $\rho$ on  $\h_t \Phi_a$ which is equal to $|dz|/(1+a)$ on $[0,K]\times [0,1+a]$ and $|dz|$ on $(K,K(1+a)] \times [0,1]$. This choice comes from the series law for extremal length: $\alpha_4$ crosses the previous two rectangles, hence its extremal length is at least the sum of theirs. Indeed, $\rho$ has area $K\left(\frac{1}{1+a}+ a\right) $ and the $\rho$-length of any arc $\gamma$ homotopic to $\alpha_4$ is at least 
$K\left(\frac{1}{1+a}+ a\right) $. Thus
\begin{equation*}
\el(\alpha_4, \g_t \Phi_a) = \el(\alpha_4, \h_t \Phi_a) \geq \frac{K^2\left( \frac{1}{1+a} + a \right)^2}{K\left(\frac{1}{1+a}+ a\right)} = K\left(\frac{1}{1+a}+ a\right). \qedhere
\end{equation*}
\end{proof}

\begin{cor} \label{cor:reverse}
If $a>0$ and $t \geq \log(1+a)$, then $\el(\alpha_1, \g_t \Phi_a) \leq \el(\alpha_4, \g_t \Phi_a)$.
\end{cor}
\begin{proof}
The condition on $t$ implies that $$1+a+e^{2t}a \leq e^{2t}\left(\frac{1}{1+a}+ a\right)$$ which gives the desired result in view of the preceding lemmata. 
\end{proof}

As indicated earlier, Proposition \ref{prop:intersection} follows from Lemma \ref{lem:equalissym}, Corollary \ref{cor:ineq}, Corollary \ref{cor:reverse} and the intermediate value theorem. By symmetry, $\eta_a$ also intersects $\gamma_2=\sigma_1(\gamma_5)$ provided that $a\geq 2$. Therefore the convex set $Q_a$ coincides with $P_t$ for some $t>0$, and this concludes the proof of Theorem \ref{thm:exhaustionpent}.

\subsection{Inner and outer radii}

It follows from Proposition \ref{prop:intersection} that the penta\-gon $Q_a$ has pe\-ri\-me\-ter at most $10 \log(1+a)$. We also want to estimate the inner and outer radii of $Q_a$ with respect to the origin.

\begin{lem} \label{lem:inner}
There exists a constant $C_1>0$ such that for every $a>0$, the pentagon $Q_a$ contains a ball of radius  $\frac{1}{2}\log a - C_1$ around the origin.
\end{lem} 

\begin{proof}
Denote the origin by $\origin$. By taking $C_1$ larger than $\frac{1}{2}\log 2$, we may assume that $a\geq 2$. In view of Proposition \ref{prop:intersection}, it suffices to show that $d(\origin, \g_t \Phi_a) \geq \frac{1}{2}\log a - C_1$ for every $t \in [0,\log(1+a)]$. By Kerckhoff's formula (\ref{eqn:kerck}) we have
$$
d(\origin, \g_t \Phi_a) \geq \frac{1}{2}\log \frac{\el(\alpha_1,\g_t \Phi_a)}{\el(\alpha_1,\origin)}.
$$ 
Let $K=e^2t$. Using the Euclidean metric on $\h_t \Phi_a$, we estimate
$$
\el(\alpha_1,\g_t \Phi_a) = \el(\alpha_1,\h_t \Phi_a) \geq \frac{(a+Ka)^2}{K(1+2a)} = \frac{(K+1)^2 a^2}{K(1+2a)} \geq \frac{4a^2}{1+2a} \geq \frac{4}{3}a
$$
where we used the inequalities $(K+1)^2 \geq 4K$ and $3a \geq 1+2a$. The result follows by taking 
\begin{equation*}C_1 \geq  \frac{1}{2}\log \frac{3 \el(\alpha_1,\origin)}{4}. \qedhere\end{equation*}

\end{proof}

\begin{lem} \label{lem:outer}
There exists a constant $C_2>0$ such that for every $a>0$, the pentagon $Q_a$ is contained in a ball of radius $\log a + C_2$ around the origin.
\end{lem}
\begin{proof}
Since $Q_a \subset Q_b$ if $a\leq b$, we may assume that $a\geq 2$. Once again, it suffices to bound $d(\origin, \g_t \Phi_a)$ from above for $t \in [0,\log(1+a)]$.  By the triangle inequality,
$$
d(\origin, \g_t \Phi_a) \leq d(\origin, \Phi_a) + d(\Phi_a, \g_t \Phi_a) \leq d(\origin, \Phi_a) + t \leq d(\origin, \Phi_a) + \log(1+a).
$$

Since $\Phi_a$ is on the ray $\gamma_1^-$, we have the equality
$$
d(\origin, \Phi_a) = \frac{1}{2}\log \frac{\el(\alpha_1, \Phi_a)}{\el(\alpha_1, \origin)}
$$
in Kerckhoff's formula. According to Lemma \ref{lem:upperalpha1}, $\el(\alpha_1, \Phi_a)  \leq 1+2a$. The result follows by combining the above inequalities with $(1+a) \leq 2a$ and $(1+2a) \leq 3a$ (recall that $a\geq 2$).
\end{proof}

\begin{cor} \label{cor:inout}
There exits a constant $C_3>0$ such that for every $t > C_3$, the pentagon $Q_a$ with $a = e^{8t/3}$ contains the ball of radius $t$ around the origin and is contained in the ball of radius $3t$ around the origin. 
\end{cor}
\begin{proof}
If $t$ is large enough then 
$$
t \leq \frac{4t}{3} - C_1 = \frac{1}{2}\log a - C_1  \quad  \text{and} \quad  \log a + C_2 = \frac{8t}{3}  + C_2 \leq 3t
$$
where $C_1$ and $C_2$ are the constants from Lemma \ref{lem:inner} and Lemma \ref{lem:outer}. The result follows from these.
\end{proof}

\subsection{Linear divergence}

Given two geodesic rays $\eta$ and $\nu$ starting from the same point $p$ in $\teich(\pentagon)$, the divergence $\divr(\eta,\nu,t)$ is defined as the distance between $\eta(t)$ and $\nu(t)$ as measured along paths disjoint from the open ball of radius $t$ centered at $p$. We can now prove that rays from the origin diverge at most linearly.

\begin{prop} \label{prop:divfromorigin}
There exists a constant $C>0$ such that for any two geodesic rays $\eta$ and $\nu$ starting from the origin in $\teich(\pentagon)$ and any $t>0$ we have
$$
\divr(\eta,\nu,t) \leq 18 t + C.
$$
\end{prop}
\begin{proof}
By adjusting the constant $C$ if necessary, it is enough to prove the inequality for $t$ large. Assume that $t>C_3$, the constant given in Corollary \ref{cor:inout}. Then the pentagon $Q_a$ with $a = e^{8t/3}$ contains the ball of radius $t$ around the origin, and is contained in the ball of radius $3t$.

We construct a path from $\eta(t)$ to $\nu(t)$ as follows. From $\eta(t)$ we continue along the same ray to reach $Q_a$ then go around $\partial Q_a$ to the intersection $x$ between $\nu$ and $\partial Q_a$ on the shortest of the two sides, then back to $\nu(t)$ along $\nu$. The constructed path has length at most twice the difference between the outer and inner radius of $Q_a$ plus half the perimeter of $Q_a$. This gives an upper bound of 
\begin{equation*}4t + 5 \log(1+e^{8t/3}) \leq  4t + \frac{40 t}{3} + \log 2 = \frac{52 t}{3} + \log 2 \leq 18t + \log 2. \qedhere\end{equation*}
\end{proof}

Using the triangle inequality, it is not hard to deduce that a similar estimate holds for rays starting from any point, which is the content of Theorem \ref{thm:divergence} for $\teich(\pentagon)$.

\begin{cor} \label{cor:triangineq}
For any $p \in \teich(\pentagon)$, there exists a constant $D>0$ such that for any geodesic rays $\eta$ and $\nu$ from $p$ and any $t>0$ we have
$$
\divr(\eta,\nu,t) \leq 18 t + D.
$$
\end{cor}
\begin{proof}
Let $\origin$ be the origin of $\teich(\pentagon)$ and let $b = d(\origin, p)$. We will show that the result holds with $D=22b + C$ where $C$ is the constant from Proposition \ref{prop:divfromorigin}. By the triangle inequality we have $$t-b \leq d(\origin , \eta(t)) \leq t + b$$  and similarly for $\nu(t)$. It follows from the intermediate value theorem that there exists some $s \in [t,t+2b]$ such that $d(\origin , \eta(s))=t+b$ and some $s' \in [t,t+2b]$ such that $d(\origin , \nu(s'))=t+b$.

We can now construct an efficient path between $\eta(t)$ and $\nu(t)$. From $\eta(t)$, we follow $\eta$ to $\eta(s)$. By Proposition \ref{prop:divfromorigin}, there is a path of length at most $18(t+b) + C$ between $\eta(s)$ and $\nu(s')$ which is disjoint from the ball $B(\origin, t+b)$, hence disjoint from $B(p,t)$. We complete the path by following $\nu$ from $\nu(s')$ to $\nu(t)$. The total length is at most \begin{equation*}2b + (18(t+b) + C) + 2b = 18t + D. \qedhere\end{equation*}
\end{proof}

Presumably, the dependence of the constant $D$ on the point $p$ can be removed (cf. \cite{DuchinRafi}), but this does not seem to follow from our methods.

Since every geodesic ray in $\teich(\pentagon)$ is Jenkins-Strebel, a result of Masur \cite{MasurThesis} implies that two geodesic rays in $\teich(\pentagon)$ stay a bounded distance apart if and only if their vertical foliations are topologically equivalent (see also \cite{Amano}). This con\-di\-tion means that if we forget the weights, then the underlying multiarc is the same. Said differently, two rays in $\teich(\pentagon)$ stay a bounded distance apart if and only if their projective vertical foliations either correspond to the same vertex or lie in the same open edge of $\pmf(\pentagon)$. Thus the divergence is often sublinear.

\section{Punctured triangles}

We prove similar results for the Teichm\"uller space $\teich(\triangletimes)$ of punctured triangles.

\subsection{Representation}

An element of $\teich(\triangletimes)$ is (an equivalence class of) a bordered Riemann surface $X$ homeo\-morphic to the closed disk together with a $4$-tuple $(x_0,x_1,x_2,x_3)$ where $x_0 \in X^\circ$ and $x_1$, $x_2$ and $x_3$ are distinct and appear in counter-clockwise order along $\partial X$. Two pairs $(X,\vec{x})$ and $(Y,\vec{y})$ are equivalent if there is a conformal diffeomorphism $h:X \to Y$ such that $h(x_j)=y_j$ for every $j\in \{0,1,2,3\}$. Again, the labelling of distinguished points plays the same role as a marking $\triangletimes \to X$.

By the Riemann Mapping Theorem, every element of $\teich(\triangletimes)$ can be represented uniquely as the closed unit disk $\overline \DD$ with $x_0 \in \DD$, $x_1=1$, $x_2=e^{2\pi/3}$ and $x_3 = e^{4\pi /3}$. With this normalization, $x_0 \in \DD$ is the only parameter. Hence $\teich(\triangletimes)$ is homeomorphic to $\DD$ or $\RR^2$.

\subsection{The three axes of symmetry}

The dihedral group $D_3$ acts on $\teich(\triangletimes)$ by permuting the labels of the boundary marked points and reversing orientation when the permutation does so. This action is isometric with respect to the Teichm\"uller metric. Let $\sigma_1=(23)$, $\sigma_2=(13)$ and $\sigma_3=(12)$. The locus $\gamma_j$ of fixed points of $\sigma_j$ is a geodesic since the quotient of $\triangletimes$ by $\sigma_j$ is a quadrilateral. If $(X,\vec{x})$ is realized as the closed unit disk with marked points $(x_0,x_1,x_2,x_3) = (x_0, 1, e^{2\pi i / 3}, e^{4\pi i / 3})$, then $\gamma_j$ is the intersection of the straight line through $0$ and $x_j$ with $\DD$. The most symmetric configuration is when $x_0=0$; we call this point the \emph{origin} of $\teich(\triangletimes)$.

\subsection{Measured foliations}

All measured foliations on the punctured triangle are tame, just like on the pentagon.

\begin{lem} \label{lem:multiarctri}
Every measured foliation on $\triangletimes$ is a weighted multiarc.
\end{lem}
\begin{proof} Let $F$ be a measured foliation on $\triangletimes$. It suffices to prove that every leaf of $F$ is a proper arc. Suppose not and let $\lambda$ be a leaf of $F$ which is recurrent to some part of $\triangletimes$. Let $\alpha$ be a short arc transverse to $F$ to which $\lambda$ returns. Starting from $\alpha$, follow $\lambda$ until it first returns to $\alpha$. The region enclosed by these arcs is a disk that possibly includes the interior marked point of $\triangletimes$. By doubling this disk across the boundary, we get a measured foliation $G$ on the sphere with at most four 1-prong singularities: at the two intersection points of $\alpha$ and $\lambda$ as well as at the interior marked point and its mirror image in the double. By the Euler--Poincar\'e formula, $G$ has exactly four $1$-prong singularities and no other singularities. This implies that $\lambda$ intersect $\alpha$ from the same side at the two intersection points, for otherwise one of these intersection points would form a $3$-prong singularity in the double. But this argument applies to all intersection points between $\lambda$ and $\alpha$, which means that they intersect only twice. Indeed, the next intersection would have to be from the other side of $\alpha$. This contradicts the hypothesis that $\lambda$ is recurrent.  
\end{proof}

\begin{figure}[htp] 
\includegraphics[scale=.8]{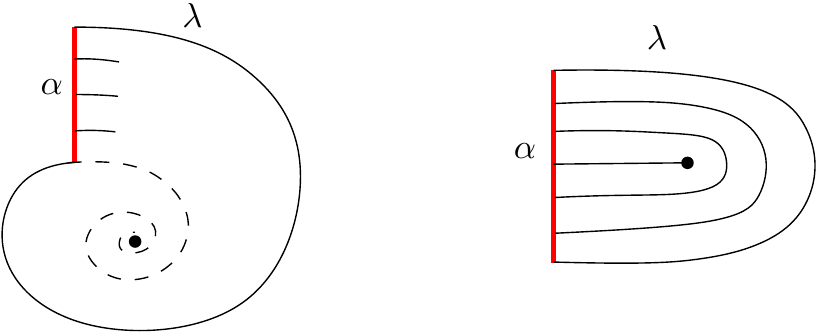}
\caption{The disk bounded by $\alpha$ and $\lambda$ in the proof of Lemma \ref{lem:multiarctri}. The situation on the left is forbidden by the Euler--Poincar\'e formula; it would force a singularity of index $-2$ at the interior marked point.} \label{fig:disk}
\end{figure}

There are two types of essential arcs in $\triangletimes$. There are those which separate two boundary marked points from the other two marked points, and those which separate the interior marked point from the $3$ boundary ones. We label the former ones by $\alpha_j$ and the latter ones by $\beta_j$ in such a way that each of $\alpha_j$ and $\beta_j$ is preserved by the reflection $\sigma_j$ (see Figure \ref{fig:pmftri}). Thought of as the arc graph, $\pmf(\triangletimes)$ is an hexagon with a bicoloring of its vertices. Indeed, the vertices $\alpha_j$ and the vertices $\beta_j$ form disjoint orbits under the action of the extended mapping class group $D_3$. 

\begin{figure}[htp] 
\includegraphics[scale=.8]{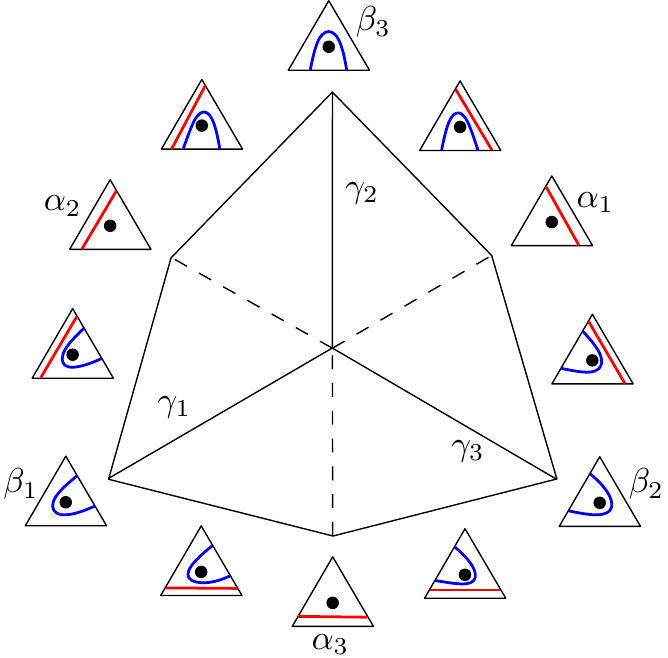}
\caption{$\pmf(\triangletimes)$ is an hexagon with $D_3$ symmetry and two types of vertices. In each small triangle, the bottom left vertex is labelled $1$ and the other vertices are labelled in counterclockwise order.} \label{fig:pmftri}
\end{figure}

\subsection{Quadratic differentials}

\begin{lem}
Every rectangular structure on $\triangletimes$ is either a rectangle or an $L$-shape with one of its horizontal segments folded in two.
\end{lem}
\begin{proof}
Let $q$ be a quadratic differential on $X \in \teich(\triangletimes)$. It is easy to see that $q$ must have a simple pole at the interior marked point $x_0$. Indeed, $q$ extends by symmetry to the double $\widetilde{X}$ of $X$, which is a sphere with 5 points marked. If $q$ did not have a pole at $x_0$, its extension $\widetilde{q}$ would have at most $3$ simple poles. The latter is forbidden by the Euler--Poincar\'e formula. Cut $X$ along the horizontal trajectory $\lambda$ from $x_0$ and call the resulting surface $Y$. Note that $x_0$ does not need to be marked in $Y$, as it unfolds to a regular boundary point (the total angle around it is $\pi$). However, the other endpoint of $\lambda$ on $\partial X$ corresponds to 2 points in $\partial Y$ which we both mark.  Thus $Y$ is a disk with $4$ or $5$ boundary marked points (depending on whether $\lambda$ ends at a marked point of $X$ or not) equipped with a rectangular structure.  The only rectangular structures on quadrilaterals are rectangles, while rectangular structures on pentagons are $L$-shapes by Lemma \ref{lem:lshape}. Since two of the marked points of $Y$ must match after folding a horizontal side, one of them must be folded exactly in two. In the case of a non-degenerate $L$-shape, the folded side must be the top or bottom one, as the inward corner is not marked. 
\end{proof}

\begin{figure}[htp] 
\includegraphics[scale=.5]{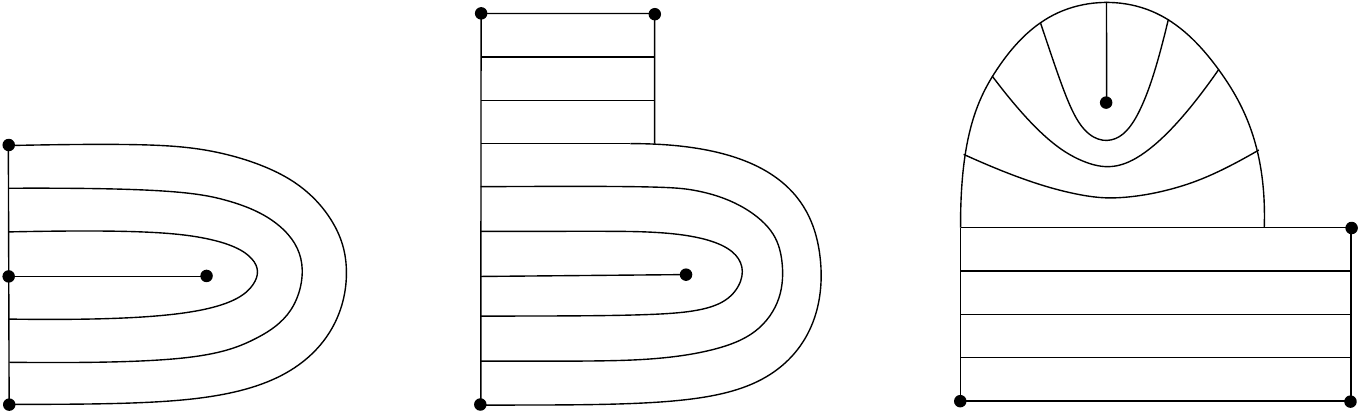}
\caption{Some examples of rectangular structures on the punctured triangle.}
\end{figure}

\subsection{Symmetric geodesics} \label{subsec:trisym}

The exact same argument as in Lemma \ref{lem:perp} applies to the current situation: $\teich(\triangletimes)$ is foliated by geodesics symmetric about $\gamma_1$. Moreover, the symmetric geodesics can be described explicitly.

Given $a \in (0,1)$, let $\Phi_a$ be the convex hull of the points $0$, $1$, $1+ia$, $a+i$ and $i$ in $\CC$ with the side $[1+ia,a+i]$ glued to itself via the central symmetry at its midpoint. The resulting object is a quadratic differential on a punctured triangle $X_a \in \teich(\triangletimes)$ with marked points $x_0 = \frac{1}{2}(1+a)(1+i)$, $x_1 = 0$, $x_2=1$ and $x_3=i$. A simple cut-and-paste procedure transforms $\Phi_a$ into an $L$-shape with a horizontal side folded in two (see Figure \ref{fig:surgery}).  The advantage of the above representation is that it is symmetric with respect to the reflection $R$ in the line $y=x$, which realizes the permutation $\sigma_1$ on the marked points. This implies that $X_a \in \gamma_1$ and that the geodesic $\eta_a = \{ \g_t \Phi_a \mid t \in \RR \}$ is symmetric about $\gamma_1$. Observe that the horizontal and vertical foliations of $\Phi_a$ are equal to $a \alpha_3 + \frac{1-a}{2} \beta_2$ and $a \alpha_2 + \frac{1-a}{2} \beta_3$ respectively.

\begin{figure}[htp] 
\includegraphics[scale=.8]{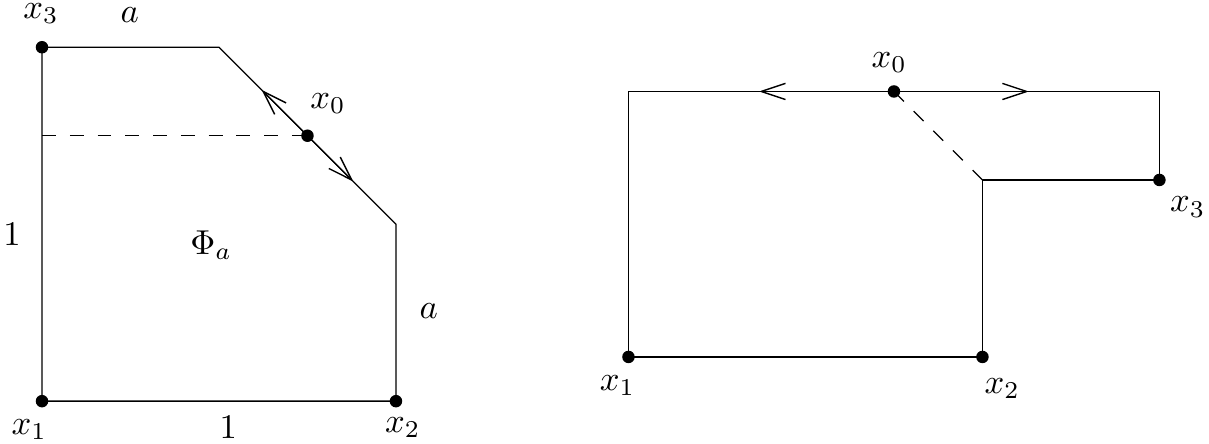}
\caption{A surgery which turns $\Phi_a$ into an $L$-shape with a horizontal side folded in two.} \label{fig:surgery}
\end{figure}

\begin{prop} \label{prop:explicit2}
Any geodesic symmetric about $\gamma_1$ in $\teich(\triangletimes)$ is equal to $\eta_a$ for a unique $a>0$.
\end{prop}

\begin{proof}
Any geodesic symmetric about $\gamma_1$ intersects $\gamma_1$ at some point $x$. Moreover, there is a unique geodesic symmetric about $\gamma_1$ through $x$. Thus we have to show that one of the geodesics $\eta_a$ passes through $x$. In other words, we have to show that the map $a \to \Phi_a$ is a bijection from $(0,1)$ to $\gamma_1$. 

Any point on $\gamma_1$ can be represented as a rectangle of unit area with vertical sides $[x_1,x_1']$ and $[x_2, x_3]$, with $x_0$ the midpoint of $[x_1,x_1']$ and that side folded in two. This rectangular structure is the Jenkins--Strebel differential for $\alpha_1$ at the corresponding point. In particular, the map $$\gamma_1(t) \mapsto \el(\alpha_1, \gamma_1(t)) = c_0 e^{- 2t}$$ is a bijection. Therefore, it suffices to prove that the map $a \mapsto \el(\alpha_1,\Phi_a)$ is a bijection.

If $a<b$, then there is a conformal embedding $\Phi_b \hookrightarrow \Phi_a$ obtained by applying a homothety of factor 
$
\frac{1+a}{1+b}
$
centered at $0$. This conformal embedding sends $x_0^b$ to $x_0^a$ and maps the sides $[x_1^b, x_2^b]$ and $[x_1^b, x_3^b]$ into the corresponding sides of $\Phi_a$. Thus every arc homotopic to $\alpha_1$ in $\Phi_b$ maps to an arc homotopic to $\alpha_1$ in $\Phi_a$. By monotonicity of extremal length under conformal embeddings, we have $\el(\alpha_1, \Phi_a) < \el(\alpha_1, \Phi_b)$ so that the above map is injective. It remains to prove surjectivity.

\begin{figure}[htp] 
\includegraphics[scale=.8]{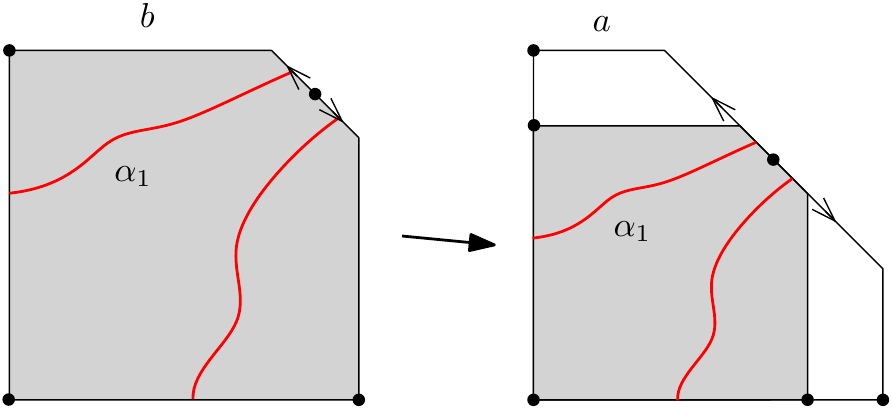}
\caption{If $a<b$, there is a conformal embedding $\Phi_b \to \Phi_a$ which sends all arcs homotopic to $\alpha_1$ on $\Phi_b$ to arcs homotopic to $\alpha_1$ on $\Phi_a$.} \label{fig:embedding}
\end{figure}

Given $a \in (0,1)$, consider the quarter annulus 
$$A_a = \left\{\, z \in \CC \ :\ 1-a <|z-(1+i)|<1 \,\right\} \cap \Phi_a.$$
Every arc homotopic to $\alpha_1$ in $\Phi_a$ has to cross $A_a$ twice (see Figure \ref{fig:annuli}). Thus 
$$
\el(\alpha_1, \Phi_a) > 2^2 \el(\text{across }A_a) = \frac{4 \log(1/(1-a))}{\pi /2}
$$
tends to $+\infty$ as $a \to 1$.

Next consider 
$$
B_a = \left\{\, z \in \CC \ :\ a \sqrt{2} <|z-(1+a)|< a \sqrt{2} + \frac{(1-a)}{\sqrt{2}} \,\right\} \cap \Phi_a
$$
and its mirror image $R(B_a)$ about the diagonal $y=x$ (see Figure \ref{fig:annuli}). These two annuli sectors glue together to form a quarter annulus $C_a = B_a \cup R(B_a)$ in $\Phi_a$. Every concentric circular arc in $C_a$ is homotopic to $\alpha_1$ so that
$$
\el(\alpha_1, \Phi_a) < \el(\text{around }C_a) = \frac{\pi/2}{ \log\left( 1+ \frac{1-a}{2a} \right) }
$$
tends to $0$ as $a \to 0$. By continuity, $\el(\alpha_1, \Phi_a)$ achieves every positive value.
\end{proof}

\begin{figure}[htp] 
\includegraphics[scale=.8]{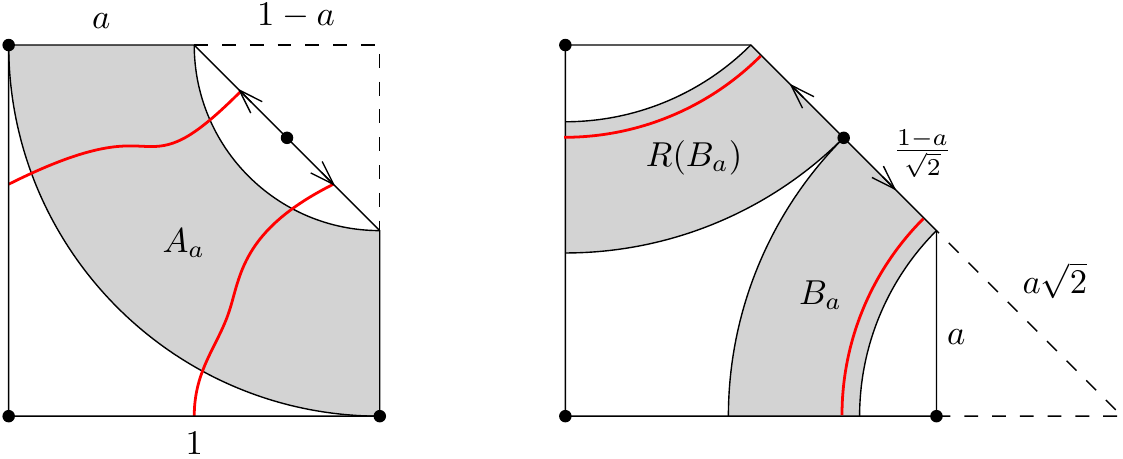}
\caption{The sectors of annuli used to bound $\el(\alpha_1, \Phi_a)$.} \label{fig:annuli}
\end{figure}

Let $U_a$ be the closed half-plane bounded by $\eta_a$ which contains the origin and let
$$
Q_a = \bigcap_{j=1}^3 \sigma_j(U_a).
$$ 
It follows from the proof of Proposition $\ref{prop:explicit2}$ that $U_a \supset U_b$ and hence $Q_a \supset Q_b$ if $0<a<b$, provided that $b$ is small enough (when $b$ passes the value $a_0$ for which $\Phi_{a_0}$ coincides with the origin, the orientation of the half-plane $U_b$ changes). Moreover,  $$\teich(\triangletimes) = \bigcup_{a \in (0,1)} Q_a$$
since the geodesics $\eta_a$ foliate the space. By construction, $Q_a$ is convex and has $D_3$ symmetry. It remains to prove that $Q_a$ is compact, i.e., that $\eta_a$ intersects $\gamma_3$.

\subsection{Equal extremal lengths implies symmetry}

We characterize the geodesic $\gamma_3$ in terms of equality of extremal lengths.

\begin{lem} \label{lem:equalissym2}
Let $X \in \teich(\triangletimes)$. The following are equivalent:
\begin{itemize}
\item $X$ belongs to $\gamma_3$;
\item $\el(\alpha_1,X) = \el(\alpha_2,X)$;
\item $\el(\beta_1,X) = \el(\beta_2,X)$.
\end{itemize}
\end{lem}
\begin{proof}
Suppose that $X \in \gamma_3$. Then there is an anti-conformal involution of $X$ realizing the permutation $\sigma_3=(12)$ on the marked points. Since $\sigma_3(\alpha_1) = \alpha_2$, $\sigma_3(\beta_1) = \beta_2$  and extremal length is invariant under anti-conformal diffeomorphisms, we have $\el(\alpha_1,X) = \el(\alpha_2,X)$ and $\el(\beta_1,X) = \el(\beta_2,X)$.

Next, we show that if $X$ is not on $\gamma_3$, then the extremal lengths of $\alpha_1$ and $\alpha_2$ are different, and similarly for $\beta_1$ and $\beta_2$. To see this, map $X$ conformally onto the unit disk is such a way that $x_0 = 0$. Let $L$ be the perpendicular bisector of the chord $[x_1,x_2]$ and let $R_L$ be the reflection in that line. Since $X \notin \gamma_3$, the point $x_3$ does not lie on $L$. Suppose that $x_3$ is closer to $x_1$ than $x_2$. Then the embedded rectangle $U$ of smallest extremal length homotopic to $\alpha_1$ maps under $R_L$ to a rectangle of the same extremal length homotopic to $\alpha_2$. Moreover, $R_L(U)$ is not extremal for $\alpha_2$ because its side contained in the circular arc from $x_2$ to $x_3$ is properly contained in that arc. Thus 
$$
\el(\alpha_1,X) = \el(U) = \el(R_L(U)) > \el(\alpha_2,X).
$$
Similarly, the embedded rectangle $V$ of smallest extremal length homotopic to $\beta_2$ maps under $R_L$ to a rectangle homotopic to $\beta_1$ which is not extremal, so that
$$
\el(\beta_2,X) = \el(V) = \el(R_L(V)) > \el(\beta_1,X).
$$
If $x_2$ is closer to $x_3$ instead, the inequalities are reversed. 
\end{proof}

Of course, the statement still holds if the indices $1$, $2$ and $3$ are permuted arbitrarily.

\subsection{Extremal length estimates}

We are ready to prove that the geodesics $\eta_a$ and $\gamma_3$ intersect if $a$ is small enough.

\begin{prop} \label{prop:intersection2}
If $a \in \left(0, \frac{1}{2e^{\pi/2} - 1}\right)$, then $\eta_a$ intersects $\gamma_3$. More precisely, $\g_t\Phi_a$ belongs to $\gamma_3$ for some $t \in \left[0,\frac12\log \frac{1}{a}\right]$.
\end{prop}

There are four inequalities to prove.

\begin{lem}
For every $a\in (0,1)$, we have 
$$
\el(\alpha_1, \Phi_a)  \leq \frac{\pi/2}{\log\left( 1+ \frac{1-a}{2a} \right)}.
$$
\end{lem}
\begin{proof}
See the proof of Proposition \ref{prop:explicit2}.
\end{proof}

\begin{lem}
For every $a\in (0,1)$, we have 
$$
\el(\alpha_2, \Phi_a)  \geq \frac{2}{\pi} \log\left( 1+ \frac{1-a}{2a} \right).
$$
\end{lem}
\begin{proof}
Every representative of $\alpha_2$ intersects every representative of $\alpha_1$ at least once. Thus every representative of $\alpha_2$ has to cross the quarter annulus $C_a$ defined in the proof of Proposition \ref{prop:explicit2}. Hence
$$
\el(\alpha_2, \Phi_a) \geq \el(\text{across }C_a) = \frac{\log\left( 1+ \frac{1-a}{2a} \right)}{ \pi /2}.
$$
This is an instance of the inequality $$\el(F, X)\el(G, X) \geq i(F,G)^2$$  due to Minsky \cite{MinskyIneq}.
\end{proof}

The next corollary follows immediately.

\begin{cor}
If $a \in \left(0, \frac{1}{2e^{\pi/2} - 1}\right)$, then $\el(\alpha_1, \Phi_a) \leq \el(\alpha_2, \Phi_a)$.
\end{cor}

We then show that the reverse inequality holds for $t$ large enough.

\begin{lem} \label{lem:loweralpha1}
For every $a\in (0,1)$ and $t\in \RR$ we have
$$
\el(\alpha_1, \g_t\Phi_a)  \geq e^{2t}a.
$$
\end{lem}
\begin{proof}
Every arc homotopic to $\alpha_1$ in $\g_t\Phi_a$ has to cross the rectangle $[0,e^{t}a]\times [0,e^{-t}]$ hori\-zontally, so the extremal length of $\alpha_1$ is at least the extremal length of that rectangle.
\end{proof}

\begin{lem}
For every $a\in (0,1)$ and $t\in \RR$ we have
$$
\el(\alpha_2, \g_t\Phi_a)  \leq \frac{1}{e^{2t}a}.
$$
\end{lem}
\begin{proof}
The vertical segments in $[0,e^{t}a]\times [0,e^{-t}]$ are homotopic to $\alpha_2$ so the extremal length of $\alpha_2$ is bounded above by the (vertical) extremal length of that rectangle.
\end{proof}

We get obtain the following as a consequence.

\begin{cor}
If $a \in (0,1)$ and $t \geq \frac12\log \frac{1}{a}$, then $\el(\alpha_1, \g_t \Phi_a) \geq \el(\alpha_2, \g_t \Phi_a)$.
\end{cor}

In turn, the two corollaries imply that $\eta_a$ intersects $\gamma_3$.

\begin{proof}[Proof of Proposition \ref{prop:intersection2}]
If $a \in \left(0, \frac{1}{2e^{\pi/2} - 1}\right)$ then  $\el(\alpha_1, g_t\Phi_a) \leq \el(\alpha_2, \g_t\Phi_a)$ at $t=0$, while the inequality is reversed at $t = \frac12\log \frac{1}{a}$. By the intermediate value theo\-rem, the equality $\el(\alpha_1, g_t\Phi_a) = \el(\alpha_2, \g_t\Phi_a)$ occurs for some $t \in \left[0,\frac12\log \frac{1}{a}\right]$. By Lemma \ref{lem:equalissym2}, equality of extremal lengths implies $\g_t\Phi_a \in \gamma_3$.
\end{proof}

Since $\eta_a$ intersects $\gamma_3$, it also intersects $\sigma_3(\eta_a)$ at the same point. By applying $\sigma_1$, we see that $\sigma_1(\eta_a)=\eta_a $ intersects $\sigma_1 \sigma_3(\eta_a)= \sigma_1 \sigma_3 \sigma_1 (\eta_a) = \sigma_2(\eta_a)$. Similarly, $\sigma_2(\eta_a)$ and $\sigma_3(\eta_a)$ intersect. Thus the intersection $Q_a$ of the corresponding half-planes $U_a$, $\sigma_2(U_a)$ and $\sigma_3(U_2)$ containing the origin is a geodesic triangle. This, together with the remarks at the end of subsection \ref{subsec:trisym}, completes the proof of Theorem \ref{thm:exhaustiontri}.

\subsection{Hexagons in the space of punctured triangles}

It turns out that the triangles $Q_a$ are bad for estimating the divergence between geodesic rays in $\teich(\triangletimes)$. Indeed, one can check that the inner radius of $Q_a$ is of order of $\log \log \frac1a$ while its outer radius and perimeter are of order $\log \frac 1a$. Following the same argument as for $\teich(\pentagon)$ would only yield that the divergence is at most exponential. But the divergence is not exponential; the triangles $\partial Q_a$ are simply inefficient paths. We replace them by more efficient hexagons.

Given $a>0$, let $\Psi_a$ be the rectangular structure on $\triangletimes$ with horizontal foliation $a \alpha_1 + \beta_2$ and vertical foliation $a \beta_3 +  \alpha_2$. We can obtain $\Psi_a$ by taking the $L$-shape $[0,1] \times [0, 1+ a] \cup [1,2(1+a)]\times [0,1]$, folding the bottom side $[0,2(1+a)]\times \{0\}$ in two, and labelling the vertices appropriately  (see Figure \ref{fig:psia}).

\begin{figure}[htp] 
\includegraphics[scale=1]{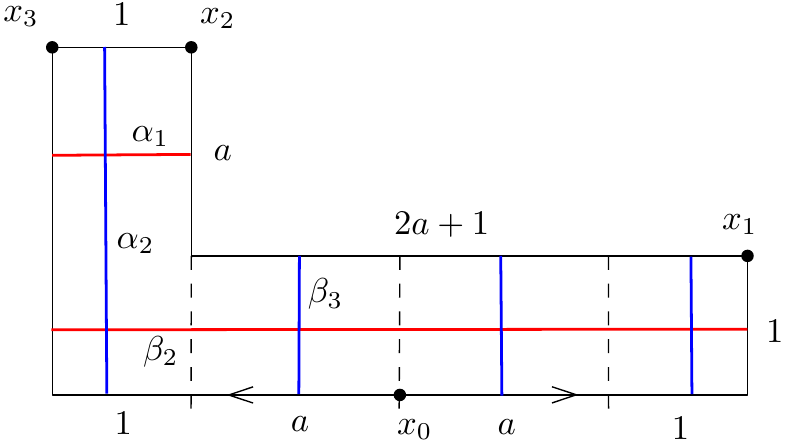}
\caption{The rectangular structure $\Psi_a$ on $\triangletimes$ with horizontal foliation $a \alpha_1 + \beta_2$ and vertical foliation $a \beta_3 +  \alpha_2$.}  \label{fig:psia}
\end{figure}

Let $\nu_a = \{ \g_t \Psi_a \mid t \in \RR \}$ be the Teichm\"uller geodesic cotangent to $\Psi_a$. We will show that $\nu_a$ intersects $\gamma_1$ and $\gamma_3$.

\begin{prop} \label{prop:intersection3}
If $a \geq 2$, then $\nu_a$ intersects $\gamma_1$ and $\gamma_3$. More precisely, $\g_t\Psi_a$ belongs to $\gamma_1$ for some $t \in \left[- \frac{1}{2}\log(2(1+a)), 0 \right]$ and $\g_t\Psi_a$ belongs to $\gamma_3$ for some $t \in \left[0 , \frac{1}{2}\log(2(1+a)) \right]$.
\end{prop}

The idea is again to estimate various extremal lengths.

\begin{lem}
If $a \geq 1$, then $\el(\beta_3,\Psi_a) \leq \el(\beta_2,\Psi_a)$.
\end{lem}
\begin{proof}
There is an $a \times 2$ rectangle embedded in $\Psi_a$ whose vertical segments are homotopic to $\beta_3$. By the second definition of extremal length we have $$\el(\beta_3,\Psi_a) \leq \frac2a.$$ The Euclidean metric on $\Psi_a$ has area $2+3a < 3(1+a)$ while any representative of $\beta_2$ has length at least $2(1+a)$. By the first definition of extremal length we have
$$
\el(\beta_2,\Psi_a) \geq \frac{(2(1+a))^2}{2+3a} > \frac43 (1+a).
$$ 
Moreover, if $a \geq 1$, then \begin{equation*}\frac2a \leq 2 < \frac83 \leq \frac43 (1+a).\qedhere\end{equation*}
\end{proof}

\begin{lem}
If $a >0$ and $t \leq - \frac12\log(2(1+a))$, then $\el(\beta_3,\g_t\Psi_a) \geq \el(\beta_2,\g_t\Psi_a)$.
\end{lem}
\begin{proof}
Let $K = e^{2t}$. The Euclidean metric on $\g_t\Psi_a$ has area $2+3a$ while any representative of $\beta_3$ has length at least $2/\sqrt{K}$. This yields
$$
\el(\beta_3,\g_t\Psi_a) \geq \frac{4}{K(2+3a)} \geq \frac{4e^{-2t}}{3(1+a)}\geq \frac 83 .
$$

On the other hand, there is a $2(1+a)\sqrt{K}$  by $1/\sqrt{K}$ rectangle homotopic to $\beta_2$ in $\g_t\Psi_a$ so that
\begin{equation*}
\el(\beta_2,\g_t\Psi_a) \leq 2(1+a)K = 2(1+a)e^{2t} \leq 1 < \frac83 \leq \el(\beta_3,\g_t\Psi_a). \qedhere
\end{equation*}
\end{proof}

\begin{cor}
If $a \geq 1$, then $\g_t\Psi_a \in \gamma_1$ for some $t \in \left[- \frac12\log(2(1+a)), 0\right]$.
\end{cor}
\begin{proof}
It follows from the previous two lemmata and the intermediate value theorem that $\el(\beta_2,\g_t\Psi_a) = \el(\beta_3,\g_t\Psi_a)$ for some $t \in \left[- \frac12\log(2(1+a)), 0\right]$. This equality implies that $\g_t\Psi_a \in \gamma_1$ by Lemma \ref{lem:equalissym2}.
\end{proof}

\begin{lem}
If $a \geq 2$, then $\el(\alpha_1,\Psi_a) \leq \el(\alpha_2,\Psi_a)$.
\end{lem}
\begin{proof}
There is a $1\times a$ rectangle homotopic to $\alpha_1$ so that $\el(\alpha_1,\Psi_a) \leq 1/a < 1$. Also, the Euclidean metric on $\Psi_a$ is such that every arc homotopic to $\alpha_2$ has length at least $2+a$. Hence we have
\begin{equation*}
\el(\alpha_2,\Psi_a) \geq \frac{(2+a)^2}{2+3a} \geq \frac{1+a}{3} \geq 1 > \el(\alpha_1,\Psi_a). \qedhere
\end{equation*}
\end{proof}

\begin{lem}
If $a\geq 2$ and $t\geq\frac12 \log(2(1+a))$, then $\el(\alpha_1,\g_t\Psi_a) \geq \el(\alpha_2,\g_t\Psi_a)$.
\end{lem}
\begin{proof}
Let $K = e^{2t}$. In the Euclidean metric on $\g_t \Psi_a$, every arc homotopic to $\alpha_1$ has length at least $\sqrt{K}$ so that
$$
\el(\alpha_1,\g_t\Psi_a) \geq \frac{K}{2+3a} \geq \frac{2(1+a)}{2+3a} \geq \frac23.
$$
Moreover, there is a $\sqrt{K}$ by $(2+a)/\sqrt{K}$ rectangle homotopic to $\alpha_2$ in $\g_t \Psi_a$, which implies
\begin{equation*}
\el(\alpha_2,\g_t\Psi_a) \leq \frac{(2+a)}{K} =  \frac{(2+a)}{e^{2t}} \leq \frac{2+a}{2(1+a)}  \leq \frac23 \leq \el(\alpha_1,\g_t\Psi_a). \qedhere
\end{equation*}
\end{proof}

\begin{cor}
If $a \geq 2$, then $\g_t\Psi_a \in \gamma_3$ for some $t \in \left[0 , \frac12\log(2(1+a))\right]$.
\end{cor}
\begin{proof}
The last two lemmata and the intermediate value theorem imply that $$\el(\alpha_1,\g_t\Psi_a) = \el(\alpha_2,\g_t\Psi_a)$$ for some $t \in \left[0 , \frac12\log(2(1+a))\right]$. In turn, equality of extremal lengths implies that $\g_t\Psi_a$ belongs to $\gamma_3$ by Lemma \ref{lem:equalissym2}.
\end{proof}

This finishes the proof of Proposition \ref{prop:intersection3}. Let $I_a$ be the segment of $\nu_a$ between $\gamma_1$ and $\gamma_3$, and let $J_a$ be the geodesic hexagon obtained by successively reflecting $I_a$ across the axes of symmetry of $\teich(\triangletimes)$:
$$
J_a = \sigma_2\sigma_1(I_a) \cup \sigma_1(I_a) \cup I_a \cup \sigma_3(I_a) \cup \sigma_2\sigma_3(I_a) \cup \sigma_1\sigma_2\sigma_3(I_a).
$$ 
Then $J_a$ is a closed curve of length at most $6\log(2(1+a))$ since $I_a$ has length at most $\log(2(1+a))$.

\subsection{Inner and outer radii}

We now estimate the inner and outer radii of the hexagon $J_a$.

\begin{lem}
There exists a constant $C_1>0$ such that for every $a \geq 2$, the hexagon $J_a$ is disjoint from the ball of radius $\frac12 \log a - C_1$ centered at the origin.
\end{lem}
\begin{proof}
Denote the origin of $\teich(\triangletimes)$ by $X_0$. It suffices to show that $$d(X_0, \g_t\Phi_a) \geq \frac12 \log a - C_1$$ whenever $|t|\leq \frac12\log(2(1+a))$. 

Let $K = e^{2t}$. In the Euclidean metric on $\g_t\Psi_a$ (which has area $3a+2 \leq 4a$), every representative of $\alpha_3$ has length at least 
$$(2a+1)\sqrt{K} + \frac{a}{\sqrt{K}} \geq a \left(\sqrt{K} + \frac{1}{\sqrt{K}}\right).$$ Thus
$$
\el(\alpha_3, \g_t\Psi_a) \geq \frac{1}{3a+2}{\left( (2a+1)\sqrt{K} + \frac{a}{\sqrt{K}}\right)^2} 
 \geq  \frac{a}{4}{\left(\sqrt{K} + \frac{1}{\sqrt{K}}\right)^2} \geq a.
$$

By Kerckhoff's formula we have
$$
d(X_0, \g_t\Psi_a) \geq \frac12 \log \frac{\el(\alpha_3, \g_t \Psi_a)}{\el(\alpha_3, X_0)} \geq \frac12 \log a - \frac12\log \el(\alpha_3, X_0).
$$
Since the last term on the right is a constant, the result follows.

\end{proof}

\begin{lem}
There exists a constant $C_2>0$ such that for every $a \geq 2$, the hexagon $J_a$ is contained in the ball of radius $\log a + C_2$ centered at the origin.
\end{lem}
\begin{proof}
Denote the origin of $\teich(\triangletimes)$ by $X_0$. It suffices to prove that the segment $I_a$ is contained in the ball, i.e., that $d(X_0, \g_t \Psi_a) \leq \log a + C_2$ whenever $|t| \leq \frac12\log(2(1+a))$. 

For every $a\geq 1$, there is a piecewise linear map $f: \Psi_1 \to \Psi_a$ obtained by stretching the top leg of $\Psi_a$ vertically by $a$ and stretching the subrectangle $[1,3]\times [0,1]$ of the right leg horizontally by $a$. The homeomorphism $f$ is $a$-quasiconformal so that $d(\Psi_1, \Psi_a) \leq \frac12 \log a$.

The triangle inequality yields the inequality
$$
d(X_0, \g_t \Psi_a) \leq d(X_0, \Psi_1) + d(\Psi_1, \Psi_a) + d(\Psi_a , \g_t \Psi_a).
$$
The first term on the right-hand side is a constant, the second term is bounded by $\frac12 \log a$ and the last term is equal to $|t|$, which is at most
\begin{equation*}
\frac12\log(2(1+a)) \leq \frac12\log(3a) = \frac12\log a +  \frac12\log 3. \qedhere
\end{equation*}
\end{proof}

\begin{cor} \label{cor:inout2}
There exits a constant $C_3>0$ such that for every $t > C_3$, the hexagon $J_a$ with $a = e^{8t/3}$ is disjoint from the ball of radius $t$ around the origin and is contained in the ball of radius $3t$ around the origin. 
\end{cor}
\begin{proof}
See the proof of Corollary \ref{cor:inout}
\end{proof}

\subsection{Linear divergence}

Since the hexagons $J_a$ have comparable inner radius, outer radius, and perimeter, it follows that geodesic rays from the origin in $\teich(\triangletimes)$ diverge at most linearly.

\begin{prop} \label{prop:divfromorigin2}
There exists a constant $C>0$ such that for any two geodesic rays $\eta$ and $\nu$ starting from the origin in $\teich(\triangletimes)$ and any $t>0$ we have
$
\divr(\eta,\nu,t) \leq 12 t + C.
$
\end{prop}
\begin{proof}
See the proof of Proposition \ref{prop:divfromorigin}. We obtain a better constant here because the half-perimeter of the hexagon $J_a$ with $a = e^{8t/3}$ is at most 
$$
3\log(2(1+a)) \leq 3 \log 3a = 8t + 3 \log 3.  
$$
to which we need to add at most $2t+2t=4t$ for joining $\eta(t)$ and $\nu(t)$ to $J_a$.
\end{proof}

By the triangle inequality, the divergence from any other point is at most linear as well.

\begin{cor} \label{cor:triangineq2}
For any $p \in \teich(\triangletimes)$, there exists a constant $D>0$ such that for any geodesic rays $\eta$ and $\nu$ from $p$ and any $t>0$ we have
$
\divr(\eta,\nu,t) \leq 12 t + D.
$
\end{cor}

This completes the proof of Theorem \ref{thm:divergence}.

\section{Universality} \label{sec:universal}

In this section, we prove Theorem \ref{thm:universal} which states that $\teich(\pentagon)$ and $\teich(\triangletimes)$ both embed isometrically in $\teich(\hexagon)$, the Teichm\"uller space of the hexagon, and that the latter embeds isometrically in the Teichm\"uller space of any closed surface of genus at least $2$.

The Teichm\"uller space $\teich(\hexagon)$ is defined analogously as for $\teich(\pentagon)$. Its points are equi\-va\-lence classes of bordered Riemann surfaces homeomorphic to the closed disk, with 6 marked points labelled in counter-clockwise order along the boundary. 

The dihedral group $D_6 \cong D_3 \times \ZZ_2$ acts isometrically on $\teich(\hexagon)$ by permuting the labels of the marked points and reversing the orientation when needed. If we take our base topological surface $\hexagon$ to be a regular hexagon in $\RR^2$, then $D_6$ acts on it by isometries. The quotient of $\hexagon$ by any of the $3$ reflections about lines through midpoints of opposite edges is a pentagon (the endpoints of the axis of reflection are critical points, hence their images have to be marked in the quotient). Each of these $3$ quotient maps is an admissible orbifold covering $\hexagon \to \pentagon$ which gives rise to an isometric embedding $\teich(\pentagon) \hookrightarrow \teich(\hexagon)$ according to Theorem \ref{thm:covering}.

Note that the $3$ copies of $\teich(\pentagon)$ obtained in this way all intersect along a single geodesic. Indeed, if an hexagon $X \in \teich(\hexagon)$ has two symmetries, it automatically has a third one. For example, if $X$ admits anti-conformal involutions acting as $\sigma = (12)(36)(45)$ and $\tau =(23)(14)(56)$ on the vertices, then it admits an anti-conformal involution acting as $\sigma\tau\sigma = (34)(25)(16)$.   

Similarly, there is a degree $2$ branched cover $\hexagon \to \triangletimes$ which we can view as the quotient of $\hexagon$ by the central symmetry about its center. This orbifold covering induces an isometric embedding $\teich(\triangletimes)\hookrightarrow \teich(\hexagon)$. Each of the $3$ copies of $\teich(\pentagon)$ in $\teich(\hexagon)$ intersects the image of $\teich(\triangletimes)$ along a geodesic. Indeed, these $3$ geodesics arise by taking the quotient of $\hexagon$ by $\ZZ_2 \times \ZZ_2$ groups, each generated by a side-to-side reflection together with the central symmetry. The quotient is a quadrilateral, whose Teichm\"uller space is isometric to $\RR$. These 3 geodesics of intersection correspond to the 3 axes of symmetry in $\teich(\triangletimes)$. See Figure \ref{fig:sketch} for a sketch of these $4$ planes sit inside $\teich(\hexagon)$. 

\begin{figure}[htp] 
\includegraphics[scale=.6]{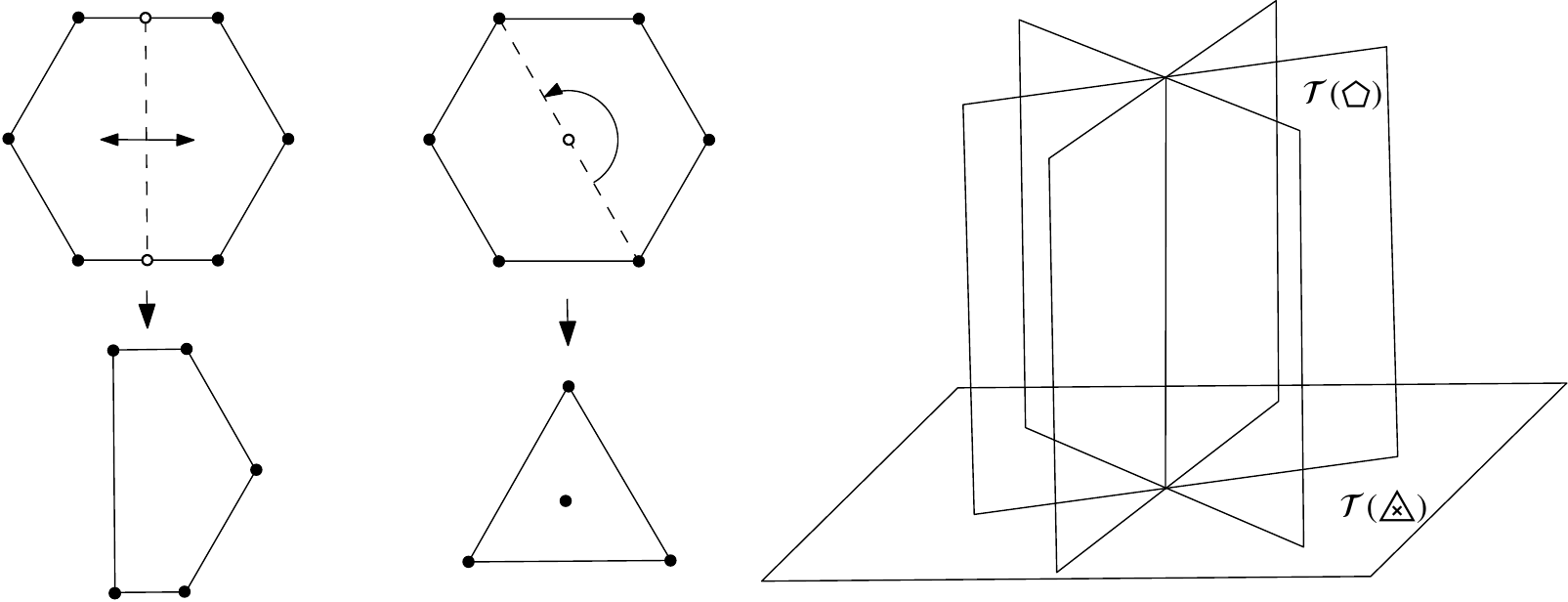}
\caption{Orbifold coverings $\hexagon \to \pentagon$ and $\hexagon \to \triangletimes$, and a sketch of how the resulting copies of $\teich(\pentagon)$ and $\teich(\triangletimes)$ sit inside $\teich(\hexagon)$.} \label{fig:sketch}
\end{figure}

Each point in $\teich(\hexagon)$ can be represented as the closed upper half-plane $\Hbar$ with marked points $x_1$, $x_2$, $x_3$, $\infty$, $-1$ and $0$, where $0<x_1<x_2<x_3$. With this normalization, the coordinate $$(\log(x_1),\log(x_2-x_1),\log(x_3-x_2))$$ provides a homeomorphism between $\teich(\hexagon)$ and $\RR^3$. Recall that each of the 3 copies of $\teich(\pentagon)$ and the copy of $\teich(\triangletimes)$ in $\teich(\hexagon)$ is the locus of fixed points of some involution in $D_6$. From this we find that they satisfy algebraic equations in the normalized coordinates $(x_1,x_2,x_3,\infty,-1,0)$:
\begin{itemize}
\item $\fix((12)(36)(45))\cong\teich(\pentagon)$ has equation $x_3 +1 = (x_1+1)(x_2 + 1) $;
\item $\fix((23)(14)(56))\cong\teich(\pentagon)$ has equation $x_1(x_1+1) = (x_2-x_1)(x_3 - x_1)$;
\item $\fix((34)(25)(16))\cong \teich(\pentagon)$ has equation $x_3(x_3-x_1) = (x_3-x_2)(x_3 + 1)$;
\item $\fix((34)(25)(16))\cong \teich(\triangletimes)$ has equation $$\left(\frac{x_3}{2}\right)^2 - \left(\frac{x_3}{2} - x_1\right)^2 = \left(\frac{x_2 + 1}{2}\right)^2 - \left(x_1 - \frac{x_2-1}{2}\right)^2.$$
\end{itemize} 
The regular hexagon corresponds to $(x_1,x_2,x_3)=(1/2,1,2)$. See Figure \ref{fig:plot} for a plot of part of these planes in log-coordinates. 

\begin{figure}[htp] 
\includegraphics[scale=.7, trim= 5cm 9cm 5cm 9.3cm, clip]{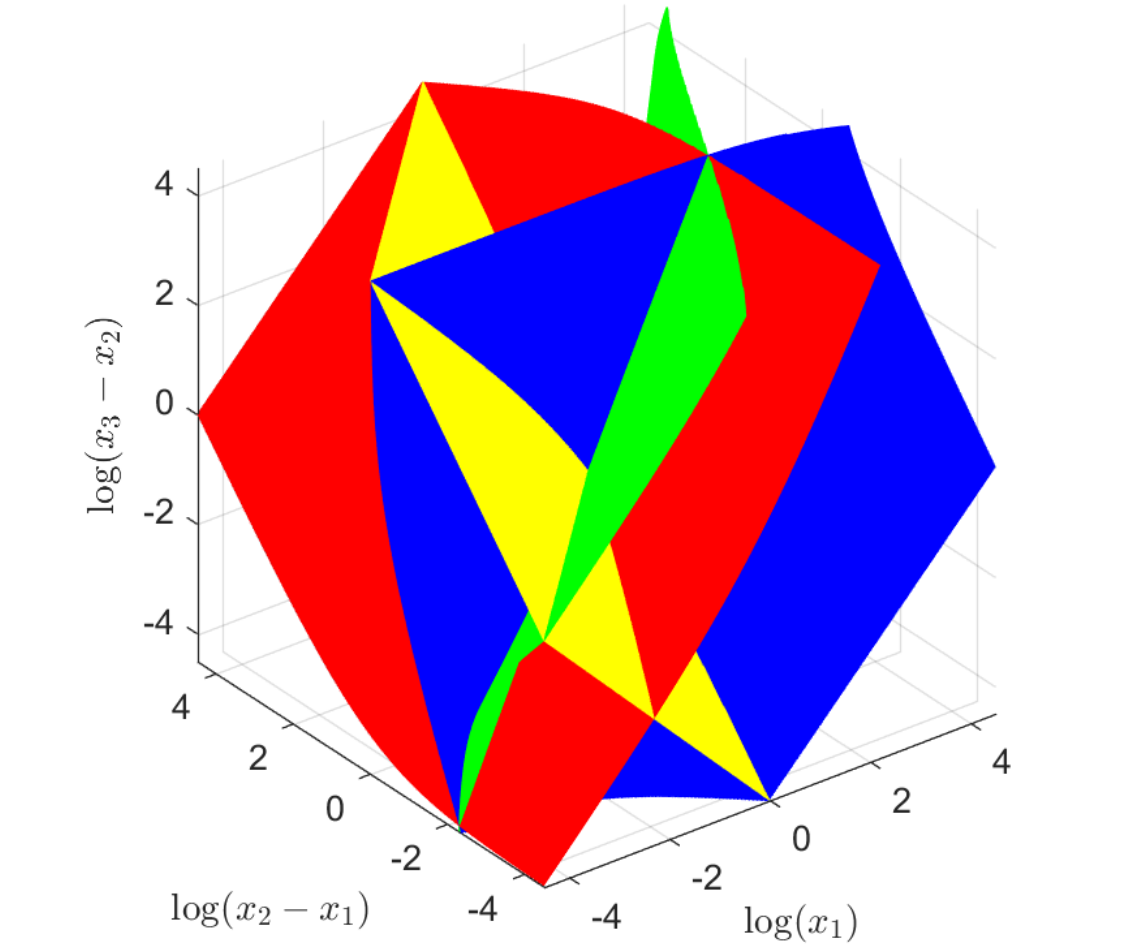}
\caption{A plot of the 3 copies of $\teich(\pentagon)$ and the copy of $\teich(\triangletimes)$ inside $\teich(\hexagon)$.} \label{fig:plot}
\end{figure}

As explained earlier, the $4$ planes described above intersect in pairs along $4$ geodesics, which we call \emph{axes of symmetry} of $\teich(\hexagon)$. In analogy with what we proved for $\teich(\pentagon)$ and $\teich(\triangletimes)$, we formulate the following conjectures:

\begin{conj}
For each of its $4$ axes of symmetry, $\teich(\hexagon)$ is foliated by totally geodesic planes invariant under the stabilizer of that axis in $D_6$. 
\end{conj}

\begin{conj}
$\teich(\hexagon)$ is a nested union of $D_6$-invariant convex triangular prisms with totally geodesic faces.
\end{conj} 

This would imply that the convex hull of any compact set in $\teich(\hexagon)$ is compact.

Back to the proof of Theorem \ref{thm:universal}. We claim that there is an isometric embedding $\teich(\hexagon) \hookrightarrow \teich(\Sigma_2)$ where $\Sigma_2$ is the closed surface of genus $2$. To see this, it suffices to give an admissible orbifold covering $\Sigma_2 \to \hexagon$. There are at least two distinct such coverings. First quotient $\Sigma_2$ by the hyper-elliptic involution to obtain a sphere with $6$ marked points, then quotient the sphere by an orientation-reversing involution fixing the 6 marked points to obtain the hexagon. Another orbifold covering is obtained as follows. First double $\hexagon$ across $3$ non-adjacent sides to get a pair of pants, then double the pair of pants across its boundary to obtain a genus $2$ surface. Reversing this process gives an orbifold covering $\Sigma_2 \to \hexagon$. Finally, it is well-known that there is a covering map $\Sigma_g \to \Sigma_2$ for every $g \geq 2$, so that $\teich(\Sigma_2)$ embeds isometrically into $\teich(\Sigma_g)$ for every $g\geq 2$ (see Figure \ref{fig:universal}).

\begin{figure}[htp] 
\includegraphics[scale=.7]{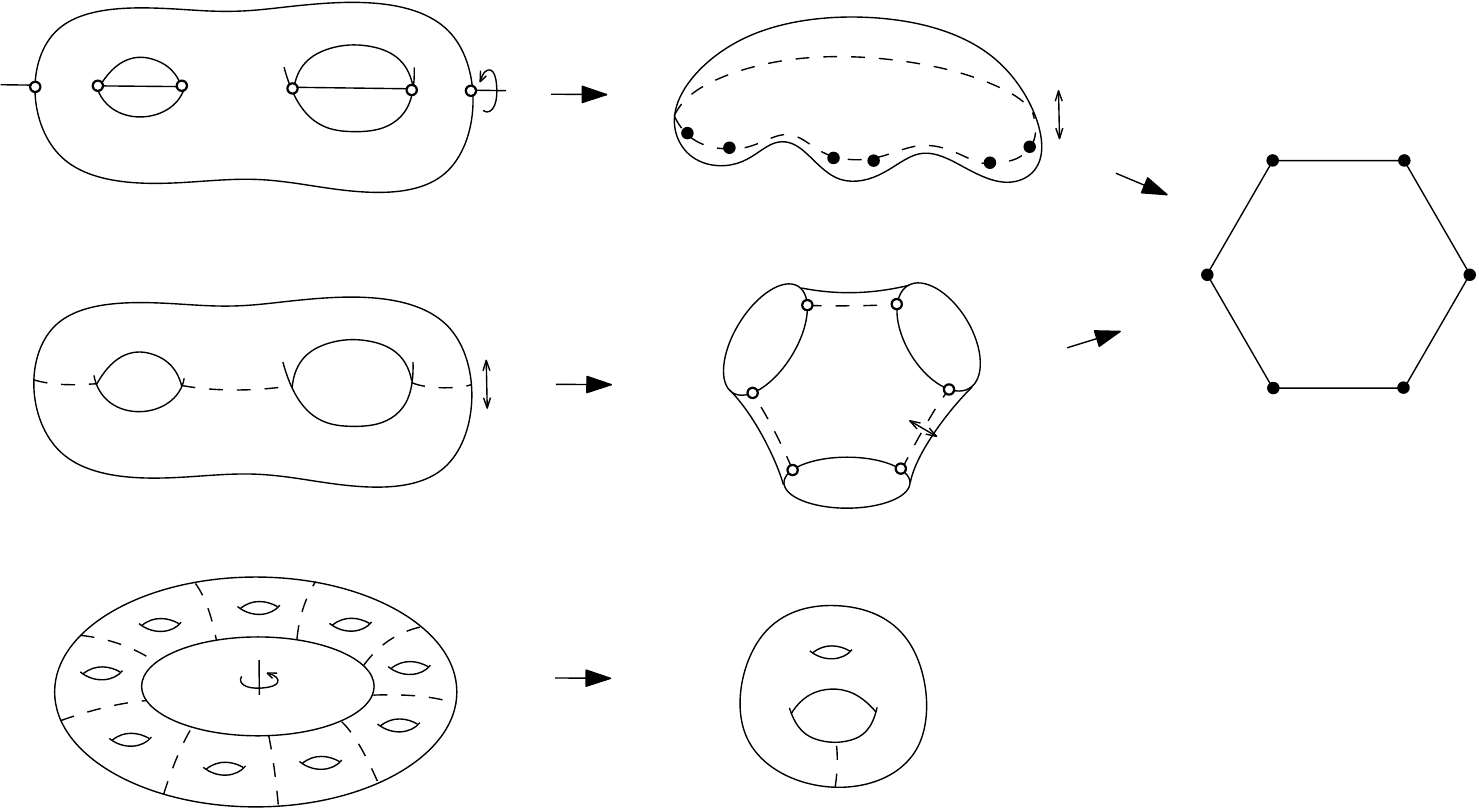}
\caption{Orbifold coverings $\Sigma_2 \to \hexagon$ and $\Sigma_g \to \Sigma_2$ for $g \geq 2$.}  \label{fig:universal}
\end{figure}

\bibliographystyle{amsalpha}

\begin{thebibliography}{MMW16}

\bibitem[Ahl06]{AhlforsLectures}
L.V. Ahlfors, \emph{Lectures on quasiconformal mappings}, University lecture
  series, American Mathematical Society, 2006.

\bibitem[Ahl10]{Ahlfors}
\bysame, \emph{Conformal invariants}, AMS Chelsea Publishing, Providence, RI,
  2010, Topics in geometric function theory, Reprint of the 1973 original, With
  a foreword by Peter Duren, F. W. Gehring and Brad Osgood.

\bibitem[Ama14]{Amano}
M.~Amano, \emph{The asymptotic behavior of {J}enkins-{S}trebel rays}, Conform.
  Geom. Dyn. \textbf{18} (2014), no.~9, 157--170.

\bibitem[Ber58]{Bers}
L.~Bers, \emph{Quasiconformal mappings and {T}eichm\"uller's theorem}, Courant
  Institute of Mathematical Sciences, New York University, 1958.

\bibitem[Bow16]{Bowditch}
B.H. Bowditch, \emph{Large-scale rank and rigidity of the {T}eichm\"uller
  metric}, J. Topology \textbf{9} (2016), no.~4, 985.

\bibitem[DR09]{DuchinRafi}
M.~Duchin and K.~Rafi, \emph{Divergence of geodesics in {T}eichm{\"u}ller space
  and the mapping class group}, GAFA \textbf{19} (2009), no.~3, 722--742.

\bibitem[FBR16]{FBRafi}
M.~Fortier~Bourque and K.~Rafi, \emph{Non-convex balls in the {T}eichm\"uller
  metric}, preprint, \href{http://arxiv.org/abs/1606.05170}{arXiv:1606.05170},
  2016.

\bibitem[FLP12]{FLP}
A.~Fathi, F.~Laudenbach, and V.~Po\'enaru, \emph{Thurston's work on surfaces},
  Mathematical Notes 48, Princeton University Press, 2012.

\bibitem[HM79]{HubbardMasur}
J.~Hubbard and H.~Masur, \emph{Quadratic differentials and foliations}, Acta
  Math. \textbf{142} (1979), 221--274.

\bibitem[Hub06]{Hubbard}
J.H. Hubbard, \emph{Teichm\"uller theory and applications to geometry,
  topology, and dynamics}, vol.~1, Matrix Editions, 2006.

\bibitem[Ker80]{Kerckhoff}
S.P. Kerckhoff, \emph{The asymptotic geometry of {T}eichm\"uller space},
  Topology \textbf{19} (1980), 23--41.

\bibitem[KPT15]{Dylan}
J.~Kahn, K.M. Pilgrim, and D.P. Thurston, \emph{Conformal surface embeddings
  and extremal length}, preprint,
  \href{http://arxiv.org/abs/1507.05294}{arXiv:1507.05294}, 2015.

\bibitem[LS14]{LeiningerSchleimer}
C.J. Leininger and S.~Schleimer, \emph{Hyperbolic spaces in {T}eichm\"uller
  spaces}, J. Eur. Math. Soc. \textbf{16} (2014), no.~12, 2669--2692.

\bibitem[Mas75]{MasurThesis}
H.~Masur, \emph{On a class of geodesics in {T}eichm\"uller space}, Ann. of
  Math. (2) \textbf{102} (1975), no.~2, 205--221.

\bibitem[Mas82]{MasurTwoBoundaries}
\bysame, \emph{Two boundaries of {T}eichm\"uller space}, Duke Math. J.
  \textbf{49} (1982), no.~1, 183--190.

\bibitem[Mas09]{MasurSurvey}
\bysame, \emph{Geometry of {T}eichm\"uller space with the {T}eichm\"uller
  metric}, Surveys in differential geometry. {V}ol. {XIV}. {G}eometry of
  {R}iemann surfaces and their moduli spaces, Surv. Differ. Geom., vol.~14,
  Int. Press, Somerville, MA, 2009, pp.~295--313.

\bibitem[Min93]{MinskyIneq}
Y.N. Minsky, \emph{Teichm\"uller geodesics and ends of hyperbolic 3-manifolds},
  Topology \textbf{32} (1993), 625--647.

\bibitem[MMW16]{MMW}
C.T. McMullen, R.E. Mukamel, and A.~Wright, \emph{Cubic curves and totally
  geodesic subvarieties of moduli space}, preprint, 2016.

\bibitem[Str84]{Strebel}
K.~Strebel, \emph{Quadratic differentials}, vol.~5, Springer-Verlag, Berlin,
  1984.

\bibitem[Tei16]{Teichmuller}
O.~Teichm\"uller, \emph{Extremal quasiconformal mappings and quadratic
  differentials}, Handbook of {T}eichm\"uller theory (A.~Papadopoulos, ed.),
  vol.~V, European Mathematical Society, 2016, pp.~321--483.

\end{thebibliography}

\providecommand{\bysame}{\leavevmode\hbox to3em{\hrulefill}\thinspace}
\providecommand{\MR}{\relax\ifhmode\unskip\space\fi MR }
\providecommand{\MRhref}[2]{%
  \href{http://www.ams.org/mathscinet-getitem?mr=#1}{#2}
}
\providecommand{\href}[2]{#2}

\end{document}